\documentclass[11pt]{amsart}
\usepackage{amsmath,amssymb,amsthm}
\usepackage{amsfonts}
\usepackage{mathrsfs}
\usepackage{mathtools}
\usepackage{multirow}
\usepackage[dvips]{graphicx}
\usepackage{color}              
\usepackage{epsfig}
\usepackage{enumerate}
\usepackage[all,cmtip]{xy}
\usepackage{float,caption}

\usepackage{tikz}
\usetikzlibrary{cd, matrix, arrows, calc, intersections, decorations.pathmorphing,backgrounds,positioning,fit, through}
\usetikzlibrary{arrows.meta}
\usetikzlibrary{calc,through,backgrounds, patterns}
\usetikzlibrary{arrows,decorations.pathmorphing, decorations.pathreplacing,backgrounds,positioning,fit}

\usepackage{pgfplots}
\usepgfplotslibrary{polar}
\pgfplotsset{compat=newest}

\usepackage[margin=1.3in]{geometry}
\newtheorem{theorem}{Theorem}[section]
\newtheorem{proposition}[theorem]{Proposition}

\newtheorem{lemma}[theorem]{Lemma}

\theoremstyle{definition}
\newtheorem{definition}[theorem]{Definition}
\newtheorem{example}[theorem]{Example}
\newtheorem*{example*}{Example}

\newtheorem*{claim*}{Claim}

\newtheorem{remark}[theorem]{Remark}
\newtheorem*{remark*}{Remark}

\newcommand{\thmref}[1]{Theorem~\ref{#1}}
\newcommand{\propref}[1]{Proposition~\ref{#1}}
\newcommand{\secref}[1]{Section~\ref{#1}}
\newcommand{\lemref}[1]{Lemma~\ref{#1}}

\newcommand{\figref}[1]{Fig.~\ref{#1}}

\newcommand{\remref}[1]{Remark~\ref{#1}}
\newcommand{\eqnref}[1]{Equation~\eqref{#1}}

\newcommand{\calG}{{\mathcal G}}

\newcommand{\calS}{{\mathcal S}}
\newcommand{\calT}{{\mathcal T}}
\newcommand{\T}{{\mathcal T}}

\newcommand{\F}{{\mathbb F}}

\newcommand{\R}{{\mathbb R}}

\newcommand{\from}{\colon \thinspace}

\newcommand{\st}{{\, \big| \,}}

\newcommand{\ep}{\epsilon}

\newcommand{\Teich}{Teich\-m\"u\-ller\ }

\newcommand{\param}{{\mathchoice{\mkern1mu\mbox{\raise2.2pt\hbox{$
\centerdot$}}
\mkern1mu}{\mkern1mu\mbox{\raise2.2pt\hbox{$\centerdot$}}\mkern1mu}{
\mkern1.5mu\centerdot\mkern1.5mu}{\mkern1.5mu\centerdot\mkern1.5mu}}}

\DeclareMathOperator{\Outer}{CV}
\DeclareMathOperator{\out}{Out}

\DeclareMathOperator{\Pre}{Pre}
\DeclareMathOperator{\CH}{CH}

\newcommand{\cv}{{\mathrm{cv}_n}}
\newcommand{\CV}{{\Outer_{\! n}}}
\newcommand{\Out}{\out(\F_n)}  
\newcommand{\Bi}{B_{\rm{in}}}
\newcommand{\Bo}{B_{\rm{out}}}

\newcommand{\be}{{\overline e}}
\newcommand{\by}{{\overline y}}
\newcommand{\bx}{{\overline x}}
\newcommand{\bphi}{{\overline \phi}}
\newcommand{\bgamma}{{\boldsymbol \gamma}}

\begin{document}

\title{Convexity of balls in the outer space}

  \author[Y.~Qing]{Yulan Qing}
  \address{Dept. of Mathematics\\ 
  University of Toronto\\
  Toronto, Ontario, Canada M5S 2E4}
  \email{\tt yulan.qing@utoronto.ca}
  
  \author[K.~Rafi]{Kasra Rafi}
  \address{Dept. of Mathematics\\ 
  University of Toronto\\
  Toronto, Ontario, Canada M5S 2E4}
  \email{\tt rafi@math.toronto.edu}
  
\date{\today}

\begin{abstract}
In this paper we study the convexity properties of geodesics and balls in Outer space
equipped with the Lipschitz metric. 
We introduce a class of geodesics called balanced folding paths and show 
that, for every loop $\alpha$, the length of $\alpha$ along a balanced folding path
is not larger than the maximum of its lengths at the end points. 
This implies that out-going balls are weakly convex. We then
show that these results are sharp by providing several counterexamples. 
\end{abstract}

\maketitle

\section{Introduction}
Let $\Out$ be the group of outer automorphisms of a free group of rank $n$,
and let Outer space $\CV$ be the space of marked metric graphs of rank $n$. 
The Outer space, which is a simplicial complex with an $\Out$ action, 
was introduced by Culler-Vogtmann \cite{outerspace} to study $\Out$ as an analogue of the 
action of mapping class group on \Teich space or the action of a lattice on a symmetric 
space. The Outer space can be equipped with a natural asymmetric metric, 
namely the Lipschitz metric. For points $x,y \in \CV$, 
\[
d(x,y) = \inf_\phi \log (L_\phi) 
\]
where $\phi \from x \to y$ is a difference of markings map from $x$ to $y$ and 
and $L_\phi$ is the Lipschitz constant of the map $\phi$. The geometry of $\CV$ equipped
with the Lipschitz metric is closely related to the large scale geometry of 
$\Out$ and has recently been the subject of extensive study (see for example, \cite{vogtmann:GO,contracting,subfactor}). 

In this paper, we examine the convexity properties of geodesics and balls in $\CV$. 
However, we need to be careful with our definitions since the metric $d$ is asymmetric 
(the ratio of $d(x,y)$ and $d(y,x)$ can be arbitrarily large \cite{asymmetry}) and there maybe
many geodesics connecting two points in $\CV$. We introduce the notion 
of a \emph{balanced folding path} along which we have more control over the lengths of 
loops. Recall that, a geodesic in $\CV$ (not necessarily parametrized with unit speed) 
is a map $\gamma \from [a,b] \to \CV$ so that, for $a\leq t \leq b$, we have  
\[
d\big(x, \gamma(t)\big) + d\big( \gamma(t), y \big) = d(x,y). 
\]
We often denote the image of $\gamma$ by $[x,y]$. The length of a loop $\alpha$ in a graph 
$x$ is denoted by $|\alpha|_x$ and we use $|\alpha|_t$ to denote the length of $\alpha$ 
at $\gamma(t)$. The balanced folding path from $x$ to $y$ is denoted by $[x,y]_{\rm bf}$.
We show that, lengths of loops along a balanced folding path satisfy a weak notion
of convexity. 

\begin{theorem} 
\label{Thm:Length} 
Given points $x, y \in \CV$, there exists a geodesic $[x,y]_{\rm bf}$ from $x$ to $y$ 
so that, for every loop $\alpha$, and every time $t$, 
\[
|\alpha|_t \leq \max \big(  |\alpha|_x, |\alpha|_y \big).
\]
\end{theorem}

The proof of this theorem is by construction. We then apply Theorem 1.1 to convexity of balls. There are two different notions of a round ball in $\CV$.  For $x \in \CV$ and 
$R>0$, we define the out-going ball of radius $R$ centered at $x$ to be
\[
\Bo (x, R) = \big\{ y \in \CV \st d(x,y) \leq R \big\}.
\]
As an immediate corollary of \thmref{Thm:Length} we have 

\begin{theorem}\label{Thm:Weakly-Convex} 
Given a point $x \in \CV$, a radius $R>0$ and points $y, z \in \Bo(x,R)$, 
\[
[y, z]_{\rm bf} \subset \Bo (x, R ).
\]
That is, the ball $ \Bo (x, R )$ is weakly convex. 
\end{theorem} 

We will also show that these theorems are sharp in various ways by providing
examples of how possible stronger statements fail. There are other ways to choose a 
geodesic connecting $y$ to $z$, for example, the \emph{standard geodesic} which
is a concatenation of a rescaling of the edges and a \emph{greedy folding path}
(see \secref{prelim} for definitions). In fact, when there is a greedy folding path 
connecting $y$ to $z$, Dowdall-Taylor \cite[Corollary 3.3]{hyperbolicextension} following 
Bestvina-Feighn \cite[Lemma 4.4]{FMladen}
have shown that the lengths of loops are quasi-convex. However, we will show
that these paths do not satisfy the conclusion of \thmref{Thm:Length} and
a standard path or a greedy folding path with end points in $\Bo(x,R)$
may leave the ball. Here is the summary of our counter-examples. 

\begin{theorem} \label{Thm:Negative} 
Theorems \ref{Thm:Length} and \ref{Thm:Weakly-Convex} are sharp. 
\begin{enumerate}
\item (Lengths cannot be made convex.) There are points $x,y \in \CV$ and a loop
$\alpha$ so that along any geodesic connecting $x$ to $y$, the length of 
$\alpha$ is not a convex function of distance in $\CV$. 

\item (The ball $\Bo$ is not quasi-convex)  This is true even when one restricts attention 
to (non-greedy)-folding paths. Namely, for any $R >0$, there are points 
$x,y,z \in \CV$ and there is a folding path $[y,z]_{\rm ng}$ connecting 
$y$ to $z$ so that 
\[
  y,z \in \Bo(x,2)
  \qquad\text{and}\qquad
  [y,z]_{\rm ng} \not \subset \Bo(x,R).
\] 
That is, a folding path with end points in $\Bo(x,2)$ can travel arbitrarily far away 
from $x$. 

\item (Standard geodesics could behave very badly) There exists a constant $c>0$ 
such that, for every $R>0$, there are points $x,y,z \in \CV$ and a standard geodesic 
$[y,z]_{\rm std}$ connecting $y$ to $z$ such that 
\[
  y,z \in \Bo(x,R)
  \qquad\text{and}\qquad
  [y,z]_{\rm std} \not \subset \Bo(x,2R-c).
\]   
That is, the standard geodesic path can travel nearly twice as far from $x$ as $y$ and $z$ 
are from $x$. 

\item (Greedy folding paths may not stay in the ball) 
For every $R>0$, there are points $x,y,z \in \CV$, $n\geq 6$, where $y$ and $z$ are 
connected by a greedy folding path $[y,z]_{\rm gf}$ such that 
\[
y,z \in \Bo(x, R) \qquad\text{but}\qquad [y,z]_{\rm gf} \not \subset \Bo(x, R). 
\]
\end{enumerate}
\end{theorem}

\subsection*{Construction of a balanced folding path} 
Given an optimal difference of markings map $\phi \from x \to y$ where the tension graph
is the whole $x$, there are many folding paths connecting $x$ to $y$. We need a controlled 
and flexible way to construct a folding path between $x$ to $y$. To this end, we introduce a 
notion of a speed assignment (see \secref{Sec:Speed}) which 
describes how fast every illegal turn in $x$ folds. Given a speed assignment, one can 
write a concrete formula for the derivative of the length of a loop $\alpha$ 
(\lemref{Length-Derivative}). 
To prove \thmref{Thm:Length}, we need to find a speed assignment so that, 
whenever $|\alpha|_y < |\alpha|_x$, the derivative of length of $\alpha$ is negative and if 
$|\alpha|_y > |\alpha|_x$ the length $\alpha$ does not grow too fast. 

A difference of markings map $\phi \from x \to y$ (again assuming the tension graph 
is the whole $x$) can be decomposed to a quotient map 
$\bphi \from x \to \by$ which is a local isometry and a scaling map $\by \to y$. 
Our approach is to determine the contribution $\ell_\tau$ of every sub-gate $\tau$ in $x$ 
to the length loss from $x$ to $\by$. For an easy example, consider 
$\F_3 =\langle a,b,c \rangle$,
let $x$ be a rose with 3 pedals where the edges are labeled $ac^2$, $bc$ and $c$
and the edge lengths are $\frac 12$, $\frac 13$ and $\frac 16$ and let $\by$ be
the rose with labels $a$, $b$ and $c$ and edges lengths all $\frac 16$. Then $\by$
is obtained from $x$ by wrapping $ac^2$ around $c$ twice and $bc$ around $c$ once. 
The length loss going from $x$ to $\by$ is 
\[
|x| - |\by| = 1-\frac 12= \frac 12. 
\]
Here, the sub-gate $\langle bc,c \rangle$ is contributing $\frac 16$ to the length loss and
the sub-gate $\langle ac^2,c \rangle$ is contributing $2 \times \frac 16$ to the length 
loss. Of course, in general, the definition of length loss contribution needs to be much 
more subtle. 

These length loss contributions are then used to determine the appropriate speed 
assignment. That is, we fold each sub-gates proportional to the length loss they eventually 
induce (see Section~\ref{weaklyconvex}). For instance, in the above example, the sub-gate 
$\langle ac^2,c \rangle$ should be folded with twice the speed of the sub-gate 
$\langle bc,c \rangle$. 

\subsection*{Decorated difference of markings map}
If the tension graph of $\phi \from x \to y$ is a proper subset of $x$, there is no
folding path between $x$ to $y$. As we see in part (2) of \thmref{Thm:Negative}, 
standard paths are not suitable for our purposes. Instead, we emulate a folding 
path even in this case. Namely, we introduce a notion of decorated difference of
markings map. That is, by adding decoration to $x$ and $y$ (marked points 
are added to $x$ and some \emph{hair} is added to $y$), we can ensure that the
difference of markings map is tight. Then, we show, a folding path can be defined
as before and the discussion above carries through (see Section~\ref{Sec:Decorated}). 

\subsection*{A criterion for uniqueness of geodesics}
To prove part one of \thmref{Thm:Negative} and \thmref{Thm:In-Ball} below, we need 
to know what all the geodesics 
connecting $x$ to $y$ are accounted for. In general this is hard to characterize. Instead, we focus on a case where there is a unique 
geodesic connecting $x$ to $y$. It is not hard to prove the uniqueness of geodesic
in special cases, however, we prove a general statement giving a criterion for uniqueness. A \emph{yo-yo} is an illegal
turn formed by a one-edge loop and a second edge, with no other edges incident to the vertex of this illegal turn.  
We say a folding path from $x$ to $y$ is \emph{rigid} if at every point along the path
there is exactly one illegal turn and it is not a yo-yo. 

\begin{theorem} \label{Thm:Unique} 
For points $x,y \in \CV$, the geodesic from $x$ and $y$ is unique if and only if there 
exists a rigid folding path connecting $x$ to $y$. 
\end{theorem} 

\subsection*{The in-coming balls} 
In the last section we examine the convexity of in-coming balls. For $x \in \CV$ and 
$R>0$, we define the in-coming ball of radius $R$ centered at $x$ to be
\[
\Bi (x, R) = \big\{ y \in \CV \st d(y,x) \leq R \big\}.
\]
We show that a ball $\Bi(x,R)$, in general, is not even weakly quasi-convex. That is, 

\begin{theorem} \label{Thm:In-Ball} 
For any constant $R>0$, there are points $x,y,z \in \CV$ such that, 
$y,z \in \Bi(x,2)$ but, for any geodesic $[y,z]$ connecting $y$ to $z$, 
\[
[y,z] \not \subset \Bi(x, R). 
\]
\end{theorem}

Once again we use \thmref{Thm:Unique}; we construct an example where there is 
a unique geodesic between $y$ and $z$ and show that it can go arbitrarily far out.

\subsection*{Analogies with \Teich space} The problem addressed in this paper has 
a long history in the setting of \Teich space. Let $(\calT(\Sigma), d_\calT)$ be the \Teich 
space of a surface $\Sigma$ equipped with the \Teich metric. It was claim by Kravetz 
\cite{kravetz:GT} that round balls in \Teich space are convex and he used it to 
give a positive answer to the Nielsen-realization problem. However, his
proof turned out to be incorrect. Even-though the Nielsen-realization problem
was solved by Kerckhoff \cite{kerckhoff:NR}, it was open for many years whether or not
the rounds balls are convex and was only resolved recently; It was shown in \cite{qc} 
that round balls in $(\calT(\Sigma), d_\calT)$ are quasi-convex and it was shown in 
\cite{rafi:NC} that there are non-convex balls in the \Teich space. The problem is still open 
for the \Teich space equipped with the Thurston metric \cite{thurston:MSM},
which is an asymmetric metric and is more directly analogous to the Lipschitz metric
in $\CV$. Hence, the weak convexity in the case of Outer space is somewhat surprising. 

The results of this paper can be used to give a positive answer the Nielsen-realization 
problem for $\CV$. However, this is already known, see \cite{khramtsov:FG,culler:FG}
and \cite{white:FF}.

\subsection*{Acknowledgements}
We would like to thank Yael Algom-Kfir and Mladen Bestvina for helpful comments 
on an earlier version of this paper.

\section{preliminaries}\label{prelim}
\subsection{Outer space}

Let $\F_n$ be a free group of rank $n$ and let $\Out$ be the outer automorphism group of
$\F_n$. Let $\cv$ be the space of free, minimal actions of $\F_n$ by isometries on metric
simplicial trees \cite{outerspace}. Two such actions are considered isomorphic if there is an equivariant 
isometry between the corresponding trees. Equivalently we think of a point in $\cv$ as the quotient metric 
graph of the tree by the corresponding action. The quotient graph is marked, that is, 
its fundamental group is identified (up to conjugation) with $\F_n$. 

The Culler-Vogtmann Outer space, $\CV$, (or simply the Outer space) is the subspace of 
$\cv$ consisting of all \textit{marked metric graphs} of total length $1$.  Let $x$ be a simple, finite graph of total length 1, in which every vertex has degree at least 3. Let $R_n$ be the graph of $n$ edges that are all incident to one vertex. A \emph{marking} is a homotopy equivalence 
$f \from R_n \to x$.   Two marked graphs $f \from R_n \to x$ and $f' \from R_n \to x'$ are \emph{equivalent} if there is an isometry $\phi \from x \to x' $ such that $\phi \circ f \simeq f'$(homotopic).  When the context is clear, we often drop 
the marking out of the notation and simply write $x \in \CV$. In this paper, we refer to metric graphs 
as $x, y$, etc. and the corresponding trees as $T_x, T_y$, etc. We also use 
$\widetilde{\phi} \from T_x \to T_y$ for the  lift of $\phi$.

The set of marked metric graphs that are 
isomorphic as marked graphs to a given point  $x \in \CV$ makes up an open simplex 
in $\CV$ which we denote $\Delta_x$. Outer space $\CV$ consists of simplices with 
missing faces. The group $\Out$ acts on $\CV$ by precomposing the marking: for an element
$g \in \Out$, $(x, f)g = (x, f\circ g) $. This is a simplicial action. 

\subsection{Lipschitz metric}
A map $\phi \from x \to y$ is a \emph{difference of markings} map if $\phi \circ f_x \simeq f_y$. We will only consider Lipschitz maps and we denote by $L_{\phi}$ the Lipschitz constant of $\phi$.  The Lipschitz metric on $\CV$ 
is defined to be:
$$
d(x, y):=\inf_\phi \log L_\phi
$$
where the infimum is taken over all differences of markings maps. There exists a non-unique 
difference of markings map that realizes the infimum\cite{FMmetric}.
Since a difference of markings map is homotopic rel vertices to a map that is linear on edges, we also use $\phi$ to 
denote the representative that is linear on edges and refer to such a map as an \emph{optimal map} from
$x$ to $y$. For the remainder of the paper, we always assume $\phi$ is an optimal difference of markings map. 

By a \emph{loop} or an \emph{immersed loop} in $x$, we mean a free homotopy class of a map from the circle
into $x$, or equivalently, a conjugacy class in $\F_n$. Meanwhile we use \emph{a simple loop}
to mean a union of edges in a graph that forms a circle with no repeated vertices.
Both kinds of loops can be identified with an element of the free group, call it $\alpha$.
We use $|\alpha|_x$ to denote the metric length 
of the shortest representative of $\alpha$ in $x$, where $x$ can be a point in $\CV$ or $\cv$, depending on the context. It is shown \cite{FMmetric} that if $x, y \in \CV$ then the distance 
$d(x, y)$ can be computed as: 
\begin{equation}\label{sup}
d(x, y) = \sup_{\alpha} \log \frac{|\alpha|_y}{|\alpha|_x},
\end{equation}
where the sup is over all loops in $x$. In fact, it is shown in \cite{FMmetric} that:
\begin{theorem}\label{candidate}
Given two points in Outer space $x$ and $y$, the immersed loop that represents $\alpha$ which realizes the supremum
can be taken from a finite set of subgraphs
of $x$ of the following forms 
\begin{itemize}
 \item simple loops
 \item figure-eight: an immersed loop where there is exactly one vertex with two pre-images in the circle
 \item dumbbell: an immersed, geodesic loop in the graph with a repeated edge
\end{itemize}
\end{theorem}
This result implies we can compute distances between two points by calculating the ratio $\frac{|\alpha|_y}{|\alpha|_x}$ for a finite
set of immersed subgraphs.

\subsection{Train track structure}
It is often convenient to use a difference of markings map that has some
additional structure. We define 
\[
\lambda(e) =  \frac{|e|_y}{|e|_x}
\]
to be the stretch factor of an edge $e$ and 
\[
\lambda(\alpha) =  \frac{|\alpha|_y}{|\alpha|_x}
\] 
to be the the stretch factor of a shortest immersed loop that represents $\alpha$.  For an optimal map $\phi \from x \to y$, define the 
\textit{tension subgraph}, $x_{\phi}$, to be the 
subgraph of $x$ consisting of maximally stretched edges. We define a 
\textit{segment} $[v, w]$ between points $v, w \in x$ to be a 
locally isometric immersion $[0, l] \rightarrow x$ of an interval
$[0, l] \subset \R$ sending $0  \rightarrow v$ and $l  \rightarrow w$. 
A \textit{direction} at a vertex $v \in x$ is a germ of non-degenerate segments $[v,w]$ with 
$v \neq w$.  $D(v)$ is the set of all directions at $v$. If a map $\phi \from x \rightarrow y$ is linear on edges with slope not 
equal to zero for all edges $e \in x$, then $\phi$ induces a derivative map $D\phi$ which 
sends a direction $d$ at $v \in x $ to a direction at $D\phi(v)$ at $\phi(v) \in y$.  
The set of \textit{gates} with respect to $\phi$ at a vertex $v \in x$ is the set of equivalence 
classes at $v$ where $d \sim d'$ if and only if $(D\phi)^k (d) = (D\phi)^k(d')$ for some 
$k \geq 1$. The size of a gate $\tau$ is denoted $|\tau|$ and is defined to be the number of directions in the equivalence class. An unordered pair $\{d,d' \}$ of distinct directions at a vertex $v$ of $x$ is called a 
\textit{turn}. The turn $\{d, d'\}$ is \textit{$\phi$--illegal} if $d$ and $d'$ belong to a same 
gate and is \textit{legal} otherwise. The set of gates at $x$ is also called the 
\textit{illegal turn structure} on $x$ induced by $\phi$. \\

\begin{definition}
A \emph{sub-gate} is a subset of directions in a gate (including the gate itself). The set 
of all sub-gates of $x \in \CV$ under the difference of markings map $\phi$ is denoted 
$\calT_{\phi}$, or simply $\calT$ if the associated map is clear from the context. 
With this terminology we can view an illegal turn as a sub-gate of size 2. 
A \emph{speed assignment} is 
an assignment of non-negative real numbers $s_{\tau}$ to all elements of $\calT_{\phi}$ of size two and denote the assignment 
\[
\calS = \Big \{ s_{\tau} \  \Big | \ \tau \in \calT, |\tau| = 2 \Big \}.
\]
Moreover, for an immersed loop $\alpha_x \in x$,
the multi-set of illegal turns of $\alpha$ is denoted $\calT_{\alpha} (\phi)$, or $\calT_{\alpha}$ when the associated map is clear from the context. $\calT_{\alpha}$ is a priori a multi-set because an illegal turn $\tau$ can appear in $\calT_{\alpha}$ more than once.
\end{definition} 

  An illegal turn structure is moreover a \textit{train track} structure if there are at least
  two gates at each vertex. This is equivalent to requiring that $\phi$ is locally injective on 
  (the interior of) each edge of $x$ and that every vertex has at least two gates. For any two points $x, y \in \CV$, there exists an optimal 
  map $\phi : x \rightarrow y$ such that $x_{\phi}$ has a train track structure\cite{FMmetric}, which then allow us 
to use immersed loops such as the candidate loops in Theorem~\ref{candidate} to compute Lipschitz distances between points.

\subsection{Folding paths} \label{Sec:Speed}
In this section we construct a family of paths called folding paths. The general definition 
can be found in \cite{FMladen}, but we introduce it here in the language that is adapted to 
this paper. Given an optimal difference of markings map $\phi \from x \to y$ with 
$x_\phi=x$, and a speed assignment $\calS = \{s_{\tau}\}_{|\tau| = 2}$, we
define a folding path $\{x_t\}$, for small $t\geq0$. The difference of markings map 
$\phi_t \from x \to x_t$ is a composition of a quotient map $\bar \phi_t \from x \to \bar x_t$
and a scaling map $\bx_t \to x_t$. For $t$ small enough, the quotient graph $\bx_t$ of $x$ 
is obtained from $x$ as follows. 
For every gate $\tau$ with $|\tau|=2$ and two points $u, w$ on the two edges $e_u, e_w$ 
of the gate $\tau$, we identify $u$ and $w$ if $|v,u|_x = |v,w|_x \leq t s_{\tau}$ 
(here $|\cdot, \cdot|$ measures the length of the segment in the graph $x$).  
The graph $\bx_t$ inherits a natural metric so that this quotient 
map is a local isometry on each edge of $x$. Since $\bx_t \in \cv$, let $x_t$ be the projective class of $\bx_t$ in $\CV$, and let $\phi_t \from x \to x_t$ be the composition $\bar \phi_t$
and the appropriate scaling. 

Notice that in $x_t$ it is possible for edges in a sub-gate $\tau$ to be identified
along a segment that is longer than $ t s_{\tau}$, depending on the identification
of other edges in the gate containing $\tau$. However we still say $x_t$ constructed 
this way is the folding path associated with $\calS = \{s_{\tau}\}_{|\tau| = 2}$.
That is, different speed assignments may result in the same folding path. 

We assume $t$ is small enough so that the combinatorial type of $x_t$ does not change on the interval $(0, t_1)$. 
Also, for $t$ small enough, any $u$ and $w$ as above are also identified under 
$\phi$ because $e_u$ and $e_w$ are in the same gate. Hence,
$ \phi \circ \phi_t^{-1}$ is a well defined map. We always assume $t$
small enough that is true and define the \emph{left-over map} $\psi_t$ at time $t$
to be defined by
\[
\psi_t \from x_t \to y, \qquad \psi_t = \phi \circ \phi_t^{-1}
\]
Note that the maps $\phi$ and $\phi_t$ have constant derivatives, therefore 
$\psi_t$ also has constant derivative and in fact the derivative is the ratio of
the derivative of $\phi$ over that of $\phi_t$. That is, 
\[
d(x_t, y) = d(x,y) -d(x,x_t). 
\]
We call such a path $\{x_t\}$ a \emph{geodesic starting from $x$ towards $y$} since
it does not necessarily reach $y$. To summarize, we have shown:

\begin{proposition} \label{Prop:Fold} 
For any difference of markings map $\phi \from x \to y$ with $x_\phi=x$ and any speed assignment $\calS=\{ s_\tau\}_{|\tau|=2}$, there is $t_1>0$ and a geodesic 
$\gamma_\calS \from [0,t_1] \to \CV$ starting from $x$ towards $y$ where the graph
$x_t = \gamma(t)$ is obtained by folding every gate $\tau$ at speed $s_{\tau}$.
\end{proposition}

Note that, when we say the combinatorial type of $x_t$ does not change, it does
not mean it is the same as $x$ or even $x_{t_1}$. Typically, the geodesic segment
$\gamma_\calG$ starts from $x$ which lies on the boundary of simplex in 
$\CV$ (not necessarily of maximal dimension) travels in the interior of this simplex
and stops when it hits the boundary of the simplex. At this point, if there is 
a new speed assignment, the folding could continue. 

In general, a folding path from $x$ to $y$ is denoted $[x, y]_{\rm f}$. If all elements of the speed assignment are equal then the path is called a \emph{greedy 
folding path} and denoted $[x, y]_{\rm gf}$. In Section 3 we construct a specific type of folding 
path whose speed assignment reflects the contribution of each sub-gate to the total length 
loss along the path. However, one has to be careful to extend the local construction described 
here to a geodesic connecting $x$ to $y$. In Section 5 we extend the local construction to a global construction.

\subsection{Standard geodesic}
Another important class of geodesic path to consider is the standard geodesic path. For two points $x, y \in \CV$, there may not exist a folding path connecting them. There is, however, a non-unique \textit{standard geodesic}, denoted $[x, y]_{\rm std}$, from $x$ to $y$ 
\cite{FMladen}. In \cite[Proposition 2.5]{FMladen}, Bestvina and Feighn give a detailed construction 
of such a standard geodesic, which we summarize briefly here. First, take an optimal map 
$\phi \from x \to y$ and consider the tension subgraph $x_{\phi}$.
Let $\Delta_x \subset \CV$ denote the smallest simplex containing $x$. By shortening some of 
the edges outside of  $x_{\phi}$ (and rescaling to maintain total length 1), one may 
then find a point $x' \in \Delta_x$ in the closed simplex $\Delta_x$ together with an optimal difference 
of markings $\phi' \from x' \to y$ whose tension graph  $x_{\phi}$ is all of 
$x'$ and such that
\[
d(x,y)=d(x,x')+d(x',y)
\]
If $\gamma_1$ denotes the linear path in $\Delta_x$ from $x$ to $x'$
(which when parameterized by arc length is a directed geodesic) and $\gamma_2 = \gamma^{\phi'}$ denotes 
the folding path from $x'$ to $y$ induced by $\phi'$, it follows from the equation above that 
the concatenation $\gamma_1 \gamma_2$ is a directed geodesic from $x$ to $y$ which is called a standard geodesic from $x$ to $y$,
which we denote $[x, y]_{\rm std}$.

\subsection{Unique geodesics}\label{unique}
We would like to show that, in certain situations, all geodesics connecting
a pair of points have some property.  Here, we give a criterion for when 
the geodesic between two points is unique. For a difference of markings map 
$\phi \from x \to y$, we define a \emph{yo-yo} to be 
an illegal turn $\langle e, \be \rangle$, at a vertex $v$ satisfying the following
(see \figref{Fig:YoYo}): 
\begin{itemize}
\item The edge $\be$ forms a loop at $v$.
\item There are no other edges attached to $v$.
\end{itemize}

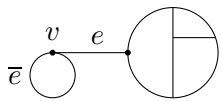
\begin{figure}[h!]
\begin{tikzpicture}
 \tikzstyle{vertex} =[circle,draw,fill=black,thick, inner sep=0pt,minimum size=.5 mm]
 
[thick, 
    scale=1,
    vertex/.style={circle,draw,fill=black,thick,
                   inner sep=0pt,minimum size= .5 mm},
                  
      trans/.style={thick,->, shorten >=6pt,shorten <=6pt,>=stealth},
   ]
   
    \node[vertex] (a) at (0,1.2) [label=$v$]  {};
    \node[vertex] (b) at (1, 1.2) {};
    \draw (0,0.9) circle (0.3cm);
    \draw (1.6,1.2) circle (0.6cm);
    \draw (0,1.2)--(1,1.2);
    \draw (1.6,.6) -- (1.6,1.8); 
    \draw (1.6,1.4) -- (2.16,1.4);     
    \node at (0.6, 1.4) {$e$};
    \node at (-.5, 0.9) {$\be$};    
\end{tikzpicture}
\caption{A yo-yo illegal turn}
\label{Fig:YoYo}
\end{figure}

We say a folding path 
$\gamma_{\rm r} \from [a,b] \to \CV$ is \emph{rigid} if, for every $t \in [a,b]$, there
is a difference of markings map $\psi_t \from \gamma_{\rm r}(t) \to y$, where the associated
train-track structure has exactly one illegal turn, and that illegal turn is not a yo-yo. 
We show that unique geodesics are exactly rigid folding paths. 

\begin{theorem}
For points $x,y \in \CV$, where $n \geq 3$, the geodesic from $x$ to $y$ is unique 
if and only if there exists a rigid folding path $\gamma_{\rm r}$ connecting $x$ to $y$.
\end{theorem}

\begin{proof}
Let $\gamma_{\rm r} \from [0,a] \to \CV$ be a rigid folding path connecting $x$ to $y$.
This implies, in particular, that there is a difference of markings map 
$\phi_{\rm r} \from x \to y$ where tension subgraph of $\phi_{\rm r}$ is all of $x$
and where the train-track structure associated to $\phi_{\rm r}$ has one, non-yo-yo illegal turn. We need the following combinatorial statement. 

\begin{claim*} 
Every edge, and every legal segment $P=\{e_1, e_2\}$ in $x$ is a subpath of an immersed 
$\phi_{\rm r}$--legal loop $\alpha$ in $x$.
\end{claim*}

\begin{proof}[Proof of Claim] \renewcommand{\qedsymbol}{$\blacksquare$} 
Let $\tau = \{ e, \be \}$ denote the only illegal turn in $x$ at vertex $v$. The graph 
$x \setminus e$ either has has a vertex of degree one, or every vertex of $x \setminus e$
has degree at least two. In the first case, $e, \be$ forms a yo-yo, which contradicts the assumption. In the second case, since every turn in $x \setminus e$ is 
$\phi_{\rm r}$--legal and every vertex has two or more gates, every edge and every length-2 legal segment is part
of an immersed legal loop. Similarly,  in $x \setminus \be$, every edge and every length-2 legal segment is part
of an immersed legal loop. Since every edge is contained either in $x \setminus e$ or in 
$x \setminus \be$, it is part of an immersed $\phi_{\rm r}$--legal loop in $x$. 

Given a legal segment $P=\{e_1, e_2\} $ of lengths 2, one of the following holds: 
\begin{itemize}
 \item $P \subset x \setminus e$
 \item $P \subset x \setminus \be$
 \item $\{e_1, e_2\} = \{e, \be \} $, and the segment starts and ends at vertex $v$
\end{itemize}

For the first two cases, we have established that $P$ is a subpath of 
an immersed $\phi_{\rm r}$--legal loop $\alpha$ in $x$. For the third case, suppose we remove $e$ and $\be$ from $x$. Since $x$ has rank 3 or higher, there is still a loop in 
$x \setminus \{ e, \be \} $, call it $\beta$. Connect $v$ and $\beta$ with 
a shortest path $p$ in $x \setminus \{ e, \be \}$. We now take $\alpha$ to be the immersed 
dumbbell shaped loop that is a concatenation of $\beta$, $\{ e, \be \} $
and two copies of $p$. This finishes the proof. 
\end{proof}

Let $z$ be a point that lies on a possibly different geodesic connecting $x$ to $y$, that is
\begin{equation}\label{Equ:unique}
d(x,z) + d(z,y) = d(x,y).
\end{equation}
Let $\phi \from x \to z$ be a difference of markings map that gives rise to a standard path
$\gamma_{\rm std}$. We decomposed $\phi=\phi_1\circ \phi_2$, where $\phi_1 \from x \to w$
represents the scaling segment of the standard path. Assume that the tension graph of 
$\phi$ is not all of $x$ and consider an edge $e \notin x_{\phi}$.

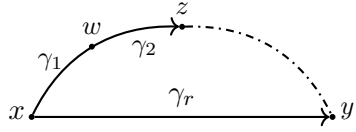
\begin{figure}[h!]

\begin{tikzpicture}
 \tikzstyle{vertex} =[circle,draw,fill=black,thick, inner sep=0pt,minimum size=.5 mm]
 
[thick, 
    scale=0.1,
    vertex/.style={circle,draw,fill=black,thick,
                   inner sep=0pt,minimum size= .5 mm},
                  
      trans/.style={thick,->, shorten >=6pt,shorten <=6pt,>=stealth},
   ]
    \node[vertex] (x) at (0, 0) [label={[xshift=-0.2cm, yshift = -0.2cm]  $x$}]  {};
    \node[vertex] (y) at (4,0) [label={[xshift=0.2cm, yshift = -0.2cm]\small $y$}] {};
    \node[vertex] (z) at ( 2, 1.2) [label={[xshift=0cm, yshift = 0cm]\small $z$}] {};
    \node[vertex]  at (0.8, 0.935)  [label={[xshift=0cm, yshift = 0cm]\small $w$}] {};
    \node at (0.25, 0.35) [label={[xshift=0cm, yshift = 0cm]\small  $\gamma_1$}]  {};
     \node at (1.5, 0.53) [label={[xshift=0cm, yshift = 0cm]\small  $\gamma_2$}]  {};

\draw[thick,->] (x) -- (y) node[midway,above,rotate=0] {$\gamma_r$};
\draw[thick, ->] (x) to [bend left = 33] (z);

\draw[thick,dash dot] (z) to [bend left = 33] (y);

          
                  \end{tikzpicture}
\end{figure}
By the claim, there exists a 
$\phi_{\rm r}$--legal immersed loop $\alpha$ containing $e$. Then
\begin{align}\label{stretch}
|\alpha|_x e^{d(x, y)} & =   \left( |\alpha|_x e^{d(x, z)} \right) e^{d(z, y)} \notag \\
                                 & < |\alpha|_z e^{d(z, y)}\\
                                 & \leq |\alpha|_y. \notag
\end{align}
But $\alpha$ is $\phi_{\rm r}$--legal, hence
\[
|\alpha|_x e^{d(x, y)} = |\alpha|_y.
\]
This is a contradiction. Thus $x_{\phi}=x$, which implies $\phi_1$ is degenerate,
$w=x$ and $\phi = \phi_2$. That is, there is a folding path $\gamma$ connecting $x$ to $z$.

Let $u=\gamma(s)$ be the first point along $\gamma$ where $\gamma_{\rm r}$ and 
$\gamma$ deviate, let $\psi_{{\rm r}, s} \from u \to y$ be the left over difference
of markings map associated to $\gamma_{\rm r}$ and $\psi_s \from u \to z$ be the left 
over difference of markings map associated to $\gamma$.  
That is, the path $\gamma \big|_{[0,s]}$ is a (possibly degenerate) sub-path of 
$\gamma_{\rm r}$, but at $u$, there is a $\psi_s$--illegal turn $\tau$ that is 
$\psi_{{\rm r}, s}$--legal. Consider the segment $P$ consists of the pair of edges
that form $\tau$. This segment is legal in $\psi_{{\rm r}, s}$, and hence by the claim, 
there exists a $\psi_{{\rm r}, s}$-legal immersed loop $\alpha$ containing $P$. 

That is, $\alpha$ is not stretching maximally from $\psi_s (s)$ to $\psi_s(s+\ep)$
and an identical to the argument to above gives a contradiction. 
Thus, the path $\psi_s$ is subpath
of $\gamma_{\rm r}$ and $z$ lies on $\gamma_{\rm r}$.

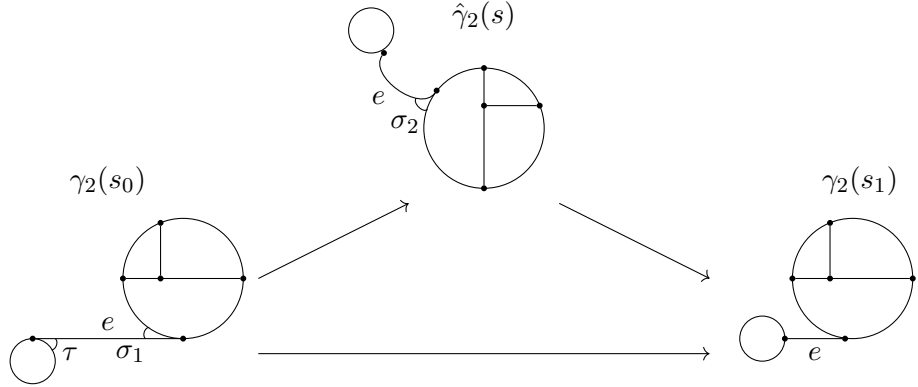
\begin{figure}[h!]
\begin{tikzpicture}
 \tikzstyle{vertex} =[circle,draw,fill=black,thick, inner sep=0pt,minimum size=.5 mm]
 
[thick, 
    scale=1,
    vertex/.style={circle,draw,fill=black,thick,
                   inner sep=0pt,minimum size= .5 mm},
                  
      trans/.style={thick,->, shorten >=6pt,shorten <=6pt,>=stealth},
   ]
   
    \node[vertex] (a) at (-1,1.2)  {};
     \node[vertex] (c) at (0.2,2)   {};
      \node[vertex] (d) at (1.8,2)  {};
        \node[vertex] (e) at (0.7,2)   {};
      \node[vertex] (f) at (0.7, 2.74)  {};
    \node[vertex] (b) at (1, 1.2)  {};
    \draw (-1,0.9) circle (0.3cm);
    \draw (1,2) circle (0.8cm);
    \draw (a)--(b);
    \draw (c)--(d);
    \draw (e)--(f);
    \draw (-0.7,1.2) to [thick, bend left = 70] (-0.75, 1.05);
    \draw (0.5,1.2) to [thick, bend left = 70] (0.55, 1.35);
    \node at (0.3, 1) {$\sigma_1$};
    \node at (-0.5, 1) {$\tau$};
    \node at (0, 1.4) {$e$};
    
      \node[vertex] (a) at (-1+5-0.33,1.2+0.8+3)  {};
       \node[vertex] (c) at (5,3.2)   {};
      \node[vertex] (d) at (5,4.8)  {};
        \node[vertex] (e) at (5,4.3)   {};
      \node[vertex] (f) at (5.74, 4.3)  {};
    \node[vertex] (b) at (5-0.8+0.5-0.33, 0.2+0.8+3+0.5)   {};
    \draw(4.1,4.4) to [bend right = 90] (4.25, 4.25);
    \node at (3.95, 4.05) {$\sigma_2$};
    \draw (4-0.5,0.3+5) circle (0.3cm);
    \draw (5,4) circle (0.8cm);
    \draw [-](a) to [bend right=90](b) ;
    \draw (c)--(d);
    \draw (e)--(f);
    \node at (3.6, 4.4) {$e$};
    
    
      \node[vertex]
      (a) at (-1+5+5,1.2)  {};
       \node[vertex] (c) at (9.1,2)   {};
      \node[vertex] (d) at (10.7,2)  {};
         \node[vertex] (e) at (9.6,2)   {};
      \node[vertex] (f) at (9.6, 2.74)  {};
    \node[vertex] (b) at (1+5-0.8-0.4+5, 1.2)  {};
    \draw (-1+5-0.3+5,0.9+0.3) circle (0.3cm);
    \draw (1+5+5-1.1,2) circle (0.8cm);
    \draw [-](a) to (b) ;
    \draw (c)--(d);
    \draw (e)--(f);
    \node at (9.4, 1) {$e$};
    
    \node at (0, 3.3) { $\gamma_2(s_0)$};
    \node at (10, 3.3) { $\gamma_2(s_1)$};
    \node at (5, 5.5) {$\hat\gamma_2(s)$};
    
    \draw [->] (2, 1) to (8,1);
    \draw [->] (2,2) to (4, 3);
    \draw [->] (6,3) to (8, 2);
\end{tikzpicture}
\caption{A yo-yo illegal turn gives rise to two geodesic paths.}
\label{nonuniqueyoyo}
\end{figure}

We now show the other direction, that is, we establish that uniqueness of a
geodesic implies that it is a rigid folding path. Consider a standard geodesic 
$[x,y]_{\rm std}$ from $x$ to $y$.  
Again, the path $[x,y]_{\rm std}$ is by definition a concatenation of a rescaling path $\gamma_1$ and a folding 
path $\gamma_2$.

Suppose that, on $\gamma_2$ there are two illegal turns at some point. Then 
\cite{FMmetric} shows that folding the two illegal turns at different speeds renders different
geodesic paths, hence obstructing uniqueness.
Otherwise, suppose $\gamma_2$ contains a yo-yo at some time $s_0$. Then a segment
$\gamma_2|_{[s_0,s_1]}$ can be replaced with a different geodesic $\hat \gamma_2$. 
Referring to \figref{nonuniqueyoyo}, the geodesic $\gamma_2|_{[s_0,s_1]}$ is obtained 
from folding the yo-yo, labeled $\tau$ from $\gamma_2(s_0)$ to $\gamma_2(s_1)$. 
However, one can also fold the edge $e$ first at $\sigma_1$ to the point $\hat\gamma_2(s)$, 
$s \in (s_0, s_1)$, and then fold $\sigma_2$ to reach $\gamma_2(s_2)=\hat\gamma_2(s_2)$.
Thus, the geodesic $\gamma_2|_{[s_0,s_1]}$ is not a unique geodesic connecting its end 
points. 

Consider now the segment $\gamma_1$. Suppose there is more than one edge that 
is not in $x_{\phi}$. Then we can choose how fast to rescale the lengths of these
edges rendering multiple geodesics with same end points as $\gamma_1$. 
Otherwise, suppose $e$ is the only edge that is not in $x_\phi$. Similar to the paths 
illustrated in \figref{nonuniqueyoyo}, $e$ can be folded onto one of its neighboring edges 
in a zigzag manner such that the end graph is isomorphic to the end point of $\gamma_1$. 
Thus $\gamma_1$ is never a unique geodesic connecting its end points, unless
it is degenerate, that is $\gamma=\gamma_2$. 

To sum up, for a standard geodesic to be unique, $\gamma_1$ is necessarily degenerate 
and $\gamma_2$ contains only one non-yo-yo illegal turn at any point. 
That is to say, it is a rigid folding path. 
\end{proof}

\begin{remark} \label{Rem:Rank-2} 
Note that, the first part of the proof still works for rank $n=2$. That is, if points $x$ and 
$y$ are connected via a rigid folding path, that path is the unique geodesic from $x$ 
to $y$. However, in ${\rm CV}_2$, even when there is a yo-yo, the geodesic
is unique. In fact, the paths $\gamma_2$ and $\hat \gamma_2$ described in 
\figref{nonuniqueyoyo} are identical in ${\rm CV}_2$. 
\end{remark}

\section{Weak convexity} \label{weaklyconvex} 

The purpose of this section is to prove the following:

\begin{theorem}\label{Maximal-Principle}
Given a difference of markings map $\phi \from x \to y$, $x, y \in \CV$, where 
$x_\phi=x$, there exists a speed assignment $\calS$ defining a folding path 
$\gamma \from [0,t_1] \to \CV$ starting at $x$ towards $y$ so that, 
for every loop $\alpha$ and every time $t \in [0,t_1]$, 
\[
|\alpha|_t \leq \max \big(  |\alpha|_x, |\alpha|_y \big).
\]
\end{theorem}

Recall that a speed assignment is a set $\calS=\{s_\tau\}_{|\tau|=2}$ of speeds
assigned to all sub-gates of size 2. 
For $t$ small enough, there is a quotient map $\bphi_t \from x \to \bar x_t$ (that is, 
$\bphi_t$ is an isometry along the edges of $x$) where the edges in gate $\tau$ are 
identified along a subsegment of length $t \, s_\tau$. Let $|\calS|$ be the speed at which
$\bar x_t$ is losing length, that is, 
\[
|\calS| = \frac{1-\text{total length of $\bar x_t$}}{t}
\]
Note that, this is a constant for small values of $t$. 

\begin{lemma} \label{Length-Derivative} 
Let $[x,y]_{\rm f}$ be a folding path associated to a difference of markings map 
$\phi \from x \to y$ and a speed assignment $\calS=\{s_\tau\}_{\tau \in \T}$. 
Then, for every loop $\alpha$, the derivative of the length of $\alpha$ along this
path equals 
\begin{equation} \label{Eq:Derivative}
\dot{|\alpha|}_t
 = |\alpha|_t - 2 \sum_{\tau \in \T_\phi(\alpha)} \frac{s_\tau}{|\calS|}
\end{equation} 
where the derivative is taken with respect to distance in $\CV$. 
\end{lemma}

\begin{proof}
For every $\tau \in \T_\alpha$, there are two sub-edges of $\alpha$ of length
$t \, s_\tau$ that are identified under the quotient map $\bphi_t \from x \to \bar x_t$.
And $x_t$ is obtained from $\bar x_t$ by a scaling of factor $\frac{1}{1- t\, |\calS|}$. Hence
\begin{equation} \label{Eq:Length-at-t} 
|\alpha|_t = 
  \frac{|\alpha|_x - 2 t \sum_{\tau \in \T_\alpha} s_\tau}{1-t \, |\calS|}
\end{equation}
and, for $s>t$, 
\[
|\alpha|_s - |\alpha|_t =  
  \frac{(s-t) \, |\calS| |\alpha|_x - 2 (s-t) \sum_{\tau \in \T_\alpha} s_\tau}
    {(1-s \, |\calS|)(1-t \, |\calS|)}.
\]
Also $d(x_t,x_s) = \log\frac{(1- s|\calS|)}{(1- t|\calS|)}$. That is, when $(s-t)$ is small, 
\[
d(x_t,x_s) =\log \left(\frac{1- t|\calS|}{1- s|\calS|} \right)
  = \log\left(1 +  \frac{(s-t)|\calS|}{1- s|\calS|} \right) \sim \frac{(s-t)\, |\calS|}{1-s|\calS|}. 
\] 
Therefore, 
\begin{equation} \label{Eq:Limit}
\dot{|\alpha|}_t = \lim_{s \to t} \frac{|\alpha_s| - |\alpha|_t }{d(x_t,x_s)} 
 = \frac{|\alpha|_x - 2 \sum_{\tau \in \T_\alpha} \frac{s_\tau}{|\calS|}}{1-t \, |\calS|}.
\end{equation}
On the other hand, replacing, $|\alpha|_t$ in the right-hand side of \eqref{Eq:Derivative} 
with the expression in \eqnref{Eq:Length-at-t}, we get 
\begin{equation} \label{Eq:Time-t}
|\alpha|_t - 2 \sum_{\tau \in \T_\phi(\alpha)} \frac{s_\tau}{|\calS|} = 
 \frac{|\alpha|_x - 2 t \sum_{\tau \in \T_\alpha} s_\tau}{1-t \, |\calS|}
- 2 \sum_{\tau \in \T_\phi(\alpha)} \frac{s_\tau}{|\calS|}
 = \frac{|\alpha|_x - 2 \sum_{\tau \in \T_\alpha} \frac{s_\tau}{|\calS|}}{1-t \, |\calS|}.
\end{equation}
The right hand sides of \eqnref{Eq:Limit} and \eqnref{Eq:Time-t} are the same, hence
the left hand sides are equal. But this is what was claimed. 
\end{proof}

For given $x$ and $y$, our goal is to find an appropriate speed assignment 
so that, for every loop $\alpha$, if $|\alpha|_y \leq |\alpha|_x$ then 
$\dot{|\alpha|}_{t=0} \leq 0$. To this end, we will define values $\ell_{\tau}$ that 
quantifies the contribution of sub-gate $\tau$ to the total length loss from $x$ to $y$
and then use these to define a speed assignment.
For the remainder of this section, let $\bar y \in \cv$ be the representative in the 
projective class of $y$ so that the associated change of marking map
$\bphi \from x \to \bar y$ restricted to every edge is a length preserving immersion. 
Also, let $\Phi \from T_x \to T_{\bar y}$ be a lift of $\bphi$.

Consider a point $p \in T_{\bar y}$ and let $\Pre(p) \subset T_x$ denote the set 
of pre-images of $p$ under the map $\Phi$ and let $\CH(p)$ denote the convex hull of 
$\Pre(p)$ in $T_x$.  
We give $\CH(p)$ a tree structure where there are no degree $2$ vertices; 
some edges of $\CH(p)$ may consist of several edges in $T_x$. 
The tree $\CH(p)$ also inherits its illegal turn structure from $T_x$, however,
an edge of $\CH(p)$ may contain one or more illegal turns. Also, note that, 
since all end points of $\CH(p)$ map to $p$, $\CH(p)$ does not contain any legal 
path connecting its end vertices. 

We denote the set of sub-gates of $T_x$ by $\Theta$. For each sub-gate 
$\theta \in \Theta$, we assign a weight $c(\sigma, p)$ to 
$\sigma$ which measures how much of the branching of $\CH(p)$ is due to 
$\sigma$. We do this inductively. In fact, we can do this for any finite subtree 
$T \subset T_x$ with the property that $T$ does not contain any legal paths connecting 
its end vertices. 

For a tree $T$, a vertex is an \emph{outer vertex} is it has degree one and an edge
is an \emph{outer edge} if one of its vertices has degree one. All other edges
are called inner edges. Assuming $T$ contains some inner edges we apply one 
of the following two steps.

\subsection*{Step 1} Consider all the outer edges of $T$. If there is an edge $e$ 
that contains illegal turns $\sigma_1, \dots, \sigma_k$, then we define 
$c(\sigma_i, p) = \frac 1k$ and we remove this edge from $T$. 
The total weight assignment is $1$ and the number of end vertices of $T$ is reduced 
by one. We observe that the remaining subtree still has the property
that it does not contain any legal paths connecting its end vertices. To see this, noticing that
since we removed exactly one topological edge, the degree of the vertex at which we moved this edge is still two or higher. 
This implies we did not create a new leaf by removing one edge, which means a legal path the would have appeared after this step
already exists before the step. But that contradicts our assumption. We apply Step 1
as many times as possible. 

\subsection*{Step 2}
We claim that there is a vertex of $T$ that has exactly two gates where one gate contains one edge, and the other gate contains only outer edges. To see this, consider the longest 
embedded path $v_0, v_1,\dots, v_m$ in $T$. Then, all but one of the edges 
connected to $v_1$ is an outer edge otherwise the path can be made longer. 
If there are more that two gates at $v_1$, or if the non-outer edge is not in its own gate, then $T$ contains a two-edge legal 
path connecting its end vertices which is a contradiction. This proves the claim. 

That is, we have found a gate $\sigma$ at a vertex $v$ where all the edges 
associated to $\sigma$ are outer edges and there is exactly one edge in the only other  
gate at $v$. We define $c(\sigma, p)=  |\sigma|-1$ and delete all these edges 
associated to $\sigma$ from $T$. Note that, the number of ends of the 
tree is reduced by $|\sigma|-1$ and the remaining tree still has the property 
the it does not contain any legal path connecting its end vertices (any such path
could be extended to a legal path in the larger tree).  

We proceed in this way, by first applying Step 1 as many times as possible, and when that is not possible Step 2.
This is a finite process since each time the number of edges in $T$ decreases. When we cannot
apply either Step 1 or Step 2, what is left is necessarily a one-gate vertex with a finite set of edges.  In this case we define
$c(\sigma, p)=  |\sigma|-1$ for that gate $\sigma$.  Again, the number of ends is one more than 
weight assigned. To summarize, the weight assigned at every step
is equal to the number of end vertices removed. We have shown 

\begin{lemma}\label{contribution}
For any point $p \in T_{\bar y}$, we have
$$
\sum_{\sigma \in \Theta} c(\sigma, p) =|\Pre(p)|-1
$$
where the sum is over all sub-gates in $T_x$. 
\end{lemma}

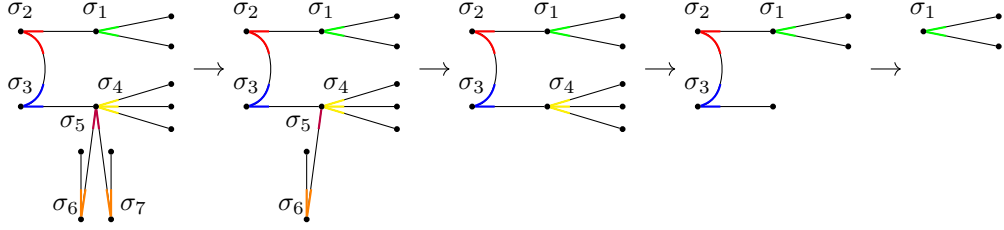
\begin{figure}
\begin{tikzpicture}
 \tikzstyle{vertex} =[circle,draw,fill=black,thick, inner sep=0pt,minimum size=.5 mm]
 
[thick, 
    scale=0.1,
    vertex/.style={circle,draw,fill=black,thick,
                   inner sep=0pt,minimum size= .5 mm},
                  
      trans/.style={thick,->, shorten >=6pt,shorten <=6pt,>=stealth},
   ]
   
    \node[vertex] (c11) at (0-3, 4) [label={[xshift=0cm, yshift = 0cm] \small  $\sigma_2$}]  {};
    \node[vertex] (c21) at (1-3, 4) [label={[xshift=0cm, yshift = 0cm]\small $\sigma_1$}] {};
    \node[vertex] (c31) at ( 2-3, 4.2)  {};
    \node[vertex] (c32) at ( 2-3, 3.8) {};
    \node [vertex] (c12) at (0-3, 3) [label={[xshift=0cm, yshift = 0cm]\small $\sigma_3$}] {};
    \node [vertex] (c22) at (1-3,3) [label={[xshift=-0.3cm, yshift = -0.5cm]\small $\sigma_5$}] [label={[xshift=0.2cm, yshift = 0cm]\small $\sigma_4$}] {};
    \node [vertex] (c33) at (2-3, 3.3) {};
    \node [vertex] (c34) at (2-3, 3) {};
    \node [vertex] (c35) at (2-3, 2.7) {};
    \node [vertex] (c23) at (0.8-3, 1.5) [label={[xshift=-0.2cm, yshift = -0.1cm]\small $\sigma_6$}] {};
    \node [vertex] (c24) at (0.8-3, 2.4) {};
    \node [vertex] (c25) at ( 1.2-3, 1.5) [label={[xshift=0.3cm, yshift = -0.1cm]\small $\sigma_7$}] {};
    \node [vertex] (c26) at (1.2-3, 2.4) {}; 
                
    \draw (c11)-- (c21)--(c31);
   
    \draw (c21)--(c32);
   
     \draw (c12)-- (c22)--(c33);
     
     \draw (c22)--(c34);
      \draw (c22)--(c35);
\draw (c22)--(c23)--(c24);
\draw (c22)--(c25) --(c26);
\draw [-] (c11) to [bend left=80] (c12);
  \draw[thick, red](c11)--(0.3-3,4);
    \draw[-,thick, red] (c11) to [bend left = 30] (0.3-3,3.7);
     \draw[thick, green](c21)--(1.3-3, 4.06);
     \draw[thick, green](c21)--(1.3-3, 3.94);
      \draw[thick, blue](c12)--(0.3-3,3);
      \draw[-,thick, blue] (c12) to [bend right = 30] (0.3-3,3.3);
\draw[thick, yellow](c22)--(1.3-3, 3);
\draw[thick, yellow](c22)--(1.3-3, 3.09);
\draw[thick, yellow](c22)--(1.3-3, 2.91);
\draw[thick, purple](c22)--(0.96-3,2.7);
\draw[thick, purple](c22)--(1.04-3,2.7);
\draw[thick, orange](c23)--(0.8-3,1.9);
\draw[thick, orange](c23)--(0.86-3, 1.9);
\draw[thick, orange](c25)--(1.2-3, 1.9);
\draw[thick, orange](c25)--(1.14-3, 1.9);

    \node[vertex] (c11) at (0, 4) [label={[xshift=0cm, yshift = 0cm] \small  $\sigma_2$}]  {};
    \node[vertex] (c21) at (1, 4) [label={[xshift=0cm, yshift = 0cm]\small $\sigma_1$}] {};
    \node[vertex] (c31) at ( 2, 4.2)  {};
    \node[vertex] (c32) at ( 2, 3.8) {};
    \node [vertex] (c12) at (0, 3) [label={[xshift=0cm, yshift = 0cm]\small $\sigma_3$}] {};
    \node [vertex] (c22) at (1,3) [label={[xshift=-0.3cm, yshift = -0.5cm]\small $\sigma_5$}] [label={[xshift=0.2cm, yshift = 0cm]\small $\sigma_4$}] {};
    \node [vertex] (c33) at (2, 3.3) {};
    \node [vertex] (c34) at (2, 3) {};
    \node [vertex] (c35) at (2, 2.7) {};
    \node [vertex] (c23) at (0.8, 1.5) [label={[xshift=-0.2cm, yshift = -0.1cm]\small $\sigma_6$}] {};
    \node [vertex] (c24) at (0.8, 2.4) {};
                
    \draw (c11)-- (c21)--(c31);
   
    \draw (c21)--(c32);
   
     \draw (c12)-- (c22)--(c33);
     
     \draw (c22)--(c34);
      \draw (c22)--(c35);
\draw (c22)--(c23)--(c24);
\draw [-] (c11) to [bend left=80] (c12);
  \draw[thick, red](c11)--(0.3,4);
    \draw[-,thick, red] (c11) to [bend left = 30] (0.3,3.7);
     \draw[thick, green](c21)--(1.3, 4.06);
     \draw[thick, green](c21)--(1.3, 3.94);
      \draw[thick, blue](c12)--(0.3,3);
      \draw[-,thick, blue] (c12) to [bend right = 30] (0.3,3.3);
\draw[thick, yellow](c22)--(1.3, 3);
\draw[thick, yellow](c22)--(1.3, 3.09);
\draw[thick, yellow](c22)--(1.3, 2.91);
\draw[thick, purple](c22)--(0.96,2.7);
\draw[thick, orange](c23)--(0.8,1.9);
\draw[thick, orange](c23)--(0.86, 1.9);

    \node[vertex] (c11) at (0+3, 4) [label={[xshift=0cm, yshift = 0cm]\small $\sigma_2$}]  {};
    \node[vertex] (c21) at (1+3, 4) [label={[xshift=0cm, yshift = 0cm]\small $\sigma_1$}] {};
    \node[vertex] (c31) at ( 2+3, 4.2)  {};
    \node[vertex] (c32) at ( 2+3, 3.8) {};
    \node [vertex] (c12) at (0+3, 3) [label={[xshift=0cm, yshift =0cm]\small $\sigma_3$}] {};
    \node [vertex] (c22) at (1+3,3)  [label={[xshift=0.2cm, yshift =0cm]\small $\sigma_4$}] {};
    \node [vertex] (c33) at (2+3, 3.3) {};
    \node [vertex] (c34) at (2+3, 3) {};
    \node [vertex] (c35) at (2+3, 2.7) {};

    \draw (c11)-- (c21)--(c31);
    \draw (c21)--(c32);
     \draw (c12)-- (c22)--(c33);
     \draw (c22)--(c34);
      \draw (c22)--(c35);
\draw [-] (c11) to [bend left=80] (c12);

  \draw[thick, red](c11)--(0.3+3,4);
    \draw[-,thick, red] (c11) to [bend left = 30] (0.3+3,3.7);
       \draw[thick, green](c21)--(1.3+3, 4.06);
     \draw[thick, green](c21)--(1.3+3, 3.94);
        \draw[thick, blue](c12)--(0.3+3,3);
      \draw[-,thick, blue] (c12) to [bend right = 30] (0.3+3,3.3);
\draw[thick, yellow](c22)--(1.3+3, 3);
\draw[thick, yellow](c22)--(1.3+3, 3.09);
\draw[thick, yellow](c22)--(1.3+3, 2.91);

    \node[vertex] (c11) at (0+6, 4) [label={[xshift=0cm, yshift =0cm]\small $\sigma_2$}]  {};
    \node[vertex] (c21) at (1+6, 4) [label={[xshift=0cm, yshift =0cm]\small $\sigma_1$}] {};
    \node[vertex] (c31) at ( 2+6, 4.2)  {};
    \node[vertex] (c32) at ( 2+6, 3.8) {};
    \node [vertex] (c12) at (0+6, 3) [label={[xshift=0cm, yshift =0cm]\small $\sigma_3$}] {};
    \node [vertex] (c22) at (1+6,3)  {};

    \draw (c11)-- (c21)--(c31);
    \draw (c21)--(c32);
     \draw (c12)-- (c22);
\draw [-] (c11) to [bend left=80] (c12);

  \draw[thick, red](c11)--(0.3+6,4);
    \draw[-,thick, red] (c11) to [bend left = 30] (0.3+6,3.7);
       \draw[thick, green](c21)--(1.3+6, 4.06);
     \draw[thick, green](c21)--(1.3+6, 3.94);
        \draw[thick, blue](c12)--(0.3+6,3);
      \draw[-,thick, blue] (c12) to [bend right = 30] (0.3+6,3.3);

    
   \node[vertex] (c21) at (1+8, 4) [label={[xshift=0cm, yshift =0cm]\small $\sigma_1$}] {};
    \node[vertex] (c31) at ( 2+8, 4.2)  {};
    \node[vertex] (c32) at ( 2+8, 3.8) {};    \draw (c21)--(c31);
    
    \draw (c21)--(c32);
      \draw[thick, green](c21)--(1.3+8, 4.06);
     \draw[thick, green](c21)--(1.3+8, 3.94);     
     \draw[->] (2.3-3,3.5) to (2.7-3,3.5);
    \draw[->] (2.3,3.5) to (2.7,3.5);
        \draw[->] (5.3,3.5) to (5.7,3.5);
            \draw[->] (8.3,3.5) to (8.7,3.5);
          
\end{tikzpicture}

\caption{ $CH(p)$ is computed iteratively by applying Step 1 as many times as possible and then apply Step 2, and then repeat.}
\label{chp}
\end{figure}

\begin{example}
The example in \figref{chp} illustrate the definition of $c(\sigma, p)$. Suppose $CH(p)$ is as shown in the leftmost graph, with seven outer edges and seven illegal turns, marked $\{ \sigma_1, \sigma_2,... \sigma_7\}$. First apply Step 1. There are two possible candidates for Step 1. The illegal turns that occur as part of a topological edge are $\sigma_6$ and $\sigma_7$, which belong to different topological edges. It does not matter which one of them is assigned first. After applying Step 1 twice, we get 
\begin{align*}
c(\sigma_6, p) = 1 \text{ and } c(\sigma_7, p) = 1. \\
\end{align*}
After deleting these two topological edges, $\sigma_5$ no longer exists, therefore
$$c(\sigma_5, p) =0. $$
Next, Step 2 picks out either $\sigma_1$ or $\sigma_4$. Suppose we start with $\sigma_4$, which contains 3 outer edges, thus 
$$
c(\sigma_4, p) = 3-1 = 2. 
$$
After deleting these three outer edges of $\sigma_4$, we apply Step 1 again and observe that there is a new topological edge with two illegal turns $\sigma_2$ and $\sigma_3$. Thus
$$
c(\sigma_2, p) = c(\sigma_3, p) = \frac{1}{2}. 
$$
And the topological edge is deleted at the end of the step. Lastly we have only one gate with two outer edge, thus 
$$
c(\sigma_1, p) = 1. 
$$
It can be verified that $$
\sum_{\sigma \in \Theta} c(\sigma, p) = 1+1+0+2+\frac{1}{2}+\frac{1}{2}+1 =6=|\Pre(p)|-1.$$
\end{example}
We now use the $c(\sigma, p)$ functions to define the length loss functions: 
\[
\ell_{\sigma} = \int_{T_{\bar y}} c(\sigma, p) \, dp
\]
where the integral is taken with respect to the length in $T_{\bar y}$. Even though
this is integral over a non-compact set, for every $\sigma$, the set of points $p$ where 
$c(\sigma, p)$ is non-zero is compact and hence the integral is finite. 
Also, since our construction is equivariant, for any sub-gate $\tau \in \T$(Recall $\calT$ is the set of all sub-gates in $x$ under the map $\phi$.) 
we can define $\ell_\tau= \ell_\sigma$ where $\sigma$ is any lift of $\tau$ to $T_x$. 

The number $\ell_{\tau}$ represent how much of the length loss from $x$ to $\bar y$ we
are attributing to the sub-gate $\tau$. In particular, we have 

\begin{lemma}\label{lengthlossfunction}
For $\bphi \from x \to \bar y$ and $\ell_\tau$ defined as above, we have 
\[
\sum_{\tau \in \T} \ell_{\tau} = |x| -|\bar y|
\]
\end{lemma}

\begin{proof}
We denote points in $\by$ by $q$ and $\Pre(q)$ represents the pre-image
of $q$ under $\bphi$. Since the map $\bphi$ is locally a length preserving immersion, 
we have
\[
1 = |x| = \int_{\bar y} |\Pre(q)| \, dq. 
\]
Let $T_0 \subset T_{\by}$ be a tree that is a fundamental domain of action $\F_n$ on
$T_{\by}$ and let $\Theta_0$ be a finite subset of $\Theta$ that contains exactly one
lift for every $\tau \in \T$. Now, 
\begin{align*}
 \sum_{\tau \in \T} \ell_{\tau}  =  \sum_{\sigma \in \Theta_0} \ell_{\sigma}  
  &=  \sum_{\sigma \in \Theta_0}  \int_{T_{\bar y}} c(\sigma, p) \, dp\\
  &=  \sum_{g \in \F_n} \sum_{\sigma \in \Theta_0}  \int_{T_0} c(\sigma, g(p)) \, dp
    \tag{$T_\by = \cup_g \, g(T_0)$}\\
  &=  \sum_{g \in \F_n} \sum_{\sigma \in \Theta_0}  \int_{T_0} c(g^{-1}(\sigma), p) \, dp
    \tag{$c(\param, \param)$ is equivariant}\\
  &= \int_{T_0}  \sum_{\sigma \in \Theta}  c(\sigma, p) \, dp
          \tag{$\Theta = \cup_g \, g(\Theta_0)$}\\
  &= \int_{T_0} |\Pre(p)| - 1 \, dp  \tag{\lemref{contribution}}\\
  &= \int_{\by} |\Pre(q)| - 1 \, dq = |x| - |\by|. 
\end{align*} 
which what was claimed in the lemma. 
\end{proof}
  
Next, we use length loss contributions $\ell_\tau$ to define a speed assignment
$\calS$. For a sub-gate $\tau$ with $|\tau|=2$, define 
\begin{equation} \label{Eq:Speed} 
 s_\tau = \sum_{\hat \tau \supseteq \tau } \frac{\ell_{\hat \tau}}{|\hat \tau|-1}
\end{equation}
where the sum is over all sub-gates $\hat \tau$ containing $\tau$. We are dividing 
$\ell_{\hat \tau}$ by $(|\hat \tau|-1)$ because if you fold edges of $\hat \tau$ along 
a segment of length $t$, the length loss is larger by factor $(|\hat \tau|-1)$. The set
$\calS= \{ s_\tau \}_{|\tau|=2}$ is our desired speed assignment. 

\begin{example}\label{Ex:introexample}        
We illustrate the computation of $s_{\tau}$ with the example mentioned in the introduction. Consider $\F_3 = \langle a, b, c \rangle$. Let $x$
be a rose with three petals. The three edges we refer to as $e_1, e_2, e_3$. The edge $e_1$ is labeled $ac^2$, the edge $e_2$ is labeled $bc$, the edge $e_3$ is labeled $c$.  The edge lengths are $ \frac{1}{2}, \frac{1}{3}, \frac{1}{6}$, respectively.  The graph $\by$
is a rose of three petals with labels $\{a, b, c\}$ and lengths $\{ \frac{1}{6},\frac{1}{6},\frac{1}{6} \} $. The construction is such that $\phi \from x \to y$ satisfies $x_{\phi} = x$. $\by$ is obtained from $x$ by wrapping $ac^2$ around c twice and $bc$ around $c$ once.

$T_{\by}$ contains three types of edges. If the point $p$ is on an edge labeled $a$ or $b$, then the pre-image contains only one copy of $p$, and $c(\tau, p) = 0$ for all $\tau$. If $p$ is on the $c$-edge, then 
$CH(p)$ is as shown in \figref{introexample}, where the four pre-images of $p$ are marked with a circle. 

\begin{figure}
\begin{tikzpicture}
 \tikzstyle{vertex} =[circle,draw,fill=black,thick, inner sep=0pt,minimum size=.3 mm]
  \tikzstyle{point} =[circle,draw=purple,fill=white,thick, inner sep=0pt,minimum size=1 mm]
 \tikzstyle{tinyvertex} =[circle,draw,fill=black,thick, inner sep=0pt,minimum size=.2 mm]
 
[thick, 
    scale=0.1,
    vertex/.style={circle,draw,fill=black,thick,
                   inner sep=0pt,minimum size= .5 mm},
                  
      trans/.style={thick,->, shorten >=6pt,shorten <=6pt,>=stealth},
   ]
    \node[vertex] (c00) at (0, 0) [label={[xshift=0cm, yshift = 0cm] \tiny  }]  {};
     \node[tinyvertex] (c01) at (0.5, 1) [label={[xshift=0cm, yshift = 0cm] \tiny  }]  {};
     \node[tinyvertex] (c02) at (1, 2) [label={[xshift=0cm, yshift = 0cm] \tiny  }]  {};

    \node[vertex] (c03) at (1.5, 3) [label={[xshift=0cm, yshift = 0cm]\tiny }] {};
    
    \node[tinyvertex] (c11) at (1.5, 1) [label={[xshift=0cm, yshift = 0cm] \tiny  }]  {};
     \node[tinyvertex] (c12) at (2, 2) [label={[xshift=0cm, yshift = 0cm] \tiny  }]  {};

    \node[vertex] (c13) at (2.5, 3) [label={[xshift=0cm, yshift = 0cm]\tiny }] {};
    \node[vertex] (c10) at ( 1, 0) [label={[xshift=0cm, yshift = 0cm] \tiny  }]  {}; {};
    \node[vertex] (c20) at ( 2, 0)[label={[xshift=0cm, yshift = 0cm] \tiny  }]  {}; {};
    \node [vertex] (c32) at (3, -1) [label={[xshift=0cm, yshift = 0cm]\tiny }] {};
     \node [tinyvertex] (c22) at (2, -0.5) [label={[xshift=0cm, yshift = 0cm]\tiny }] {};

    \draw (c00)--(c03);
    \draw(c00)--(c10)--(c20);
        \draw(c10)--(c11)--(c12)--(c13);
    \draw(c10)--(c32);
    
    \node at (0.25,0.2) [label={[xshift=-0.1cm, yshift = 0cm] \tiny  $c$}]  {};
     \node at (0.75,1.2) [label={[xshift=-0.1cm, yshift = 0cm] \tiny  $c$}]  {};
      \node at (1.25,2.2) [label={[xshift=-0.1cm, yshift = 0cm] \tiny  $a$}]  {};
      
      \node at (1.25,0.2) [label={[xshift=-0.1cm, yshift = 0cm] \tiny  $c$}]  {};
     \node at (1.75,1.2) [label={[xshift=-0.1cm, yshift = 0cm] \tiny  $c$}]  {};
      \node at (2.25,2.2) [label={[xshift=-0.1cm, yshift = 0cm] \tiny  $a$}]  {};
      
       \node at (0.5, -0.2) [label={[xshift=0cm, yshift = 0cm] \tiny  $c$}]  {};
       \node at (1.5, -0.2) [label={[xshift=0cm, yshift = 0cm] \tiny  $c$}]  {};
       
       \node at(1.3, -0.7)[label={[xshift=0cm, yshift = 0cm] \tiny  $c$}]  {};
        \node at(2.6, -1)[label={[xshift=0cm, yshift = 0cm] \tiny  $b$}]  {};

\draw[thin, green] (c00)--(c10)--(c20) ;
\draw[thin, blue] (c10)--(c32);

 \node [point] at (0.85, 1.7) [label={[xshift=0cm, yshift = 0cm]}] {};
          \node [point] at (1.35, 0.7) [label={[xshift=0cm, yshift = 0cm]}] {};
               \node [point] at (1.7, 0) [label={[xshift=0cm, yshift = 0cm]}] {};
                    \node [point] at (1.7, -0.35) [label={[xshift=0cm, yshift = 0cm]}] {};

\draw [-] (0.05,0.1) to [bend left=80] (0.1, 0);
\node at(-0.1, -0.55) [label={[xshift=0cm, yshift = 0cm] \tiny  $\sigma_{e_1e_3}$}]  {};
\draw [-] (1.05,0.1) to [bend left=80] (1.1, -0.05);
\node at(0.9, -0.55) [label={[xshift=0cm, yshift = 0cm] \tiny  $\sigma_{e_1e_2e_3}$}]  {};

                  \end{tikzpicture}
                  \caption{The $CH(p)$ of Example~\ref{Ex:introexample} where $p$ is a point on the edge labeled $c$ in $y$}
                  \label{introexample}
                  \end{figure}

$CH(p)$ has two gates. One contains a black and green edge, which we denote $\sigma_{e_1e_3}$. The other gate contains three edges, black, green and blue, and we denote the gate $\sigma_{e_1e_2e_3}$. At Step 1, $$c(\sigma_{e_1e_3}, p) = 1;$$ 
at Step 2, $$\sigma_{e_1e_2e_3} =2.$$

Next we compute the length loss function by integrating $c(\param, \param)$ over $T_{\by}$. In this case, the only non-zero component of the integral is when integrating over edge labeled $c$, which has length $\frac{1}{6}$, therefore:
\[
\ell_{e_1e_3}= 1 \times \frac{1}{6} = \frac {1}{6} \]
\[ \ell_{e_1e_2e_3} = 2 \times \frac{1}{6} = \frac {1}{3}.
\]
Indeed, it is the case that 
$$\frac{1}{3} + \frac{1}{6}  = \sum_{\sigma} \ell_{\sigma} = x - \by = 1 - \frac{1}{2} = \frac{1}{2}$$
Based on the length loss functions we compute the folding speed of all sub-gates in $x$. Again we can denote a sub-gate in $x$ by the edges in the gate, so we have 
$$
s_{e_1e_3} = l_{e_1e_3} + \frac{1}{2} \ell_{e_1e_2e_3} = \frac{1}{6} + \frac{1}{2} \times \frac{1}{3} = \frac{1}{3}$$

$$
s_{e_1e_2} = s_{e_2e_3}  =  \frac{1}{2} \ell_{e_1e_2e_3} =\frac{1}{2} \times \frac{1}{3} = \frac{1}{6}$$

That is to say, since $ac^2$ wraps over $c$ twice while $bc$ wraps over $c$ once, infinitesimally, the folding associated to the former is twice as fast. 
\end{example}
 
\begin{lemma}\label{lengthlossspeed}
For the speed assignment $\calS$ above, we have
\[
|\calS| \leq \sum_{\tau \in \T} \ell_\tau.  
\]
\end{lemma} 

\begin{proof}
We organize the argument by considering one gate $\tau$ and its sub-gates
only. The length losses at different gates add up, hence, it is sufficient to prove
the lemma one gate at the time. For the rest of the argument, let $\tau$ be
a fixed gate in $\T$. 

By an $\ep$--neighbourhood of a gate $\tau$ we mean the intersection of 
an $\ep$--ball around the vertex associated to $\tau$ with the edges associated 
to $\tau$. For $\ep$ small enough, the $\ep$--neighbourhood of $\tau$ is a tree with one
vertex $v$ and $|\tau|$ edges $e_1, \dots, e_{|\tau|}$ of size $\ep$. We denote
the sub-gate consisting of edges $e_i$ and $e_j$ by $\tau_{i,j}$ and 
use a shorthand $s_{i,j}$ to denote $s_{\tau_{i,j}}$. Choose $t>0$ small enough so that, 
for every $1 \leq i, j \leq |\tau|$, $t \, s_{i,j} < \ep$. 

The image of this $\ep$--neighbourhood in $\overline{x}_t$ is a quotient of
the $\ep$--neighbourhood of $\tau$ after identifying $e_i$ and $e_j$ along 
a segment of length $s_{i,j}$ starting from the vertex $v$. Therefore, the total 
length loss at $\tau$, $|\calS_\tau|$, has the following upper-bound
\begin{equation} \label{Eq:Loss-Estimate} 
|\calS_\tau| \leq \sum_{i=1}^{|\tau|}  \sum_{j=i+1}^{|\tau|} s_{i,j}. 
\end{equation}
We have an inequality here because some identifications maybe redundant; if 
a segment of $e_1$ is identified with both $e_2$ and $e_3$, identifying $e_2$ and $e_3$
along this subsegment does not cause any further length loss. We re-organize the above 
sum as follows: let $\calT_\tau^i$ be the set of sub-gates $\tau'$ of $\tau$ where $e_i$
is the edge in $\tau'$ with the smallest index. We claim
\begin{equation} \label{Eq:Re-Organize}
\sum_{j=i+1}^{|\tau|} s_{i,j} 
 = \sum_{j=i+1}^{|\tau|} \sum_{\tau' \supseteq \tau_{i,j} } \frac{l_{\tau'}}{|\tau'|-1}
 = \sum_{\tau' \in \T_\tau^i} \ell_{\tau'}. 
\end{equation}
This is because, for every $\tau' \subseteq \tau$, the term $\frac{l_{\tau'}}{|\tau'|-1}$ 
appears exactly $(|\tau'|-1)$--times in the above sum, once for every $e_j \in \tau'$ 
where $e_j \not = e_i$. 
Combining \eqnref{Eq:Loss-Estimate} and \eqnref{Eq:Re-Organize}, we have
\[
|\calS_\tau| \leq \sum_{i=1}^{|\tau|} 
    \sum_{\tau' \in \T_\tau^i} \ell_{\tau'} 
  = \sum_{\tau' \subseteq \tau} \ell_{\tau'}. 
\]
This finishes the proof. 
\end{proof}
 
\begin{lemma} \label{Vanishing-Path} 
Let $v \subset T_x$ be a vanishing path and let $\Theta_v$ be the set of sub-gates in $T_x$ that appear along $v$.
\[
|v|_x \leq 2 \sum_{\sigma \in \Theta_v} s_{\sigma}. 
\]
\end{lemma} 

\begin{proof}
Since $v$ is a vanishing path, $\Phi$ sends its end-points to some point $p \in T_\by$. 
We claim
\[
1 \leq \sum_{\sigma \in \Theta_v} \sum_{\hat \sigma \supseteq \sigma} 
    \frac{c(\hat \sigma,p)}{|\hat \sigma|-1} . 
\]
This is because $v$ is contained in $\CH(p)$. If $v$ passes through an edge $e$
of $\CH(p)$ with an illegal turn, then the sum 
\[
1 = \sum_{\sigma \in \Theta_e} c(\sigma, p) \leq \sum_{\sigma \in \Theta_v} c(\sigma, p) 
\]
and the claim follows. Otherwise, a sub-gates $\sigma \in \Theta_v$ is contained
in a sub-gate $\bar \sigma$ that appears in $\CH(p)$, where all the associated edges 
are legal, and hence, following the algorithm, $c(\hat \sigma, p) = |\hat \sigma|-1$. 
Again the claim follows. Note that the map $\Phi$ from $v$ to $\Phi(v)$ is 2-to-1. Hence
\begin{align*}
|v|_x &=  \int_{\Phi(v)} 2 \, dp \\
  &\leq 2 \int_{\Phi(v)}  \sum_{\sigma \in \Theta_v} \sum_{\hat \sigma \supseteq \sigma}
     \frac{c(\hat \sigma,p)}{|\hat \sigma|-1}  \, dp 
      \tag{Using the claim}\\
 & \leq 2  \sum_{\sigma \in \Theta_v} \sum_{\hat \sigma \supseteq \sigma}
      \int_{T_\by}  \frac{c(\hat \sigma,p)}{|\hat \sigma|-1}  \, dp 
      \tag{Enlarging the domain of integration} \\
 & =  2  \sum_{\sigma \in \Theta_v} \sum_{\hat \sigma \supseteq \sigma}
      \frac{\ell_\sigma}{|\hat \sigma|-1} = 2 \sum_{\sigma \in \Theta_v} s_{\sigma}. 
\end{align*}
And we are done. 
\end{proof}

\begin{proof}[Proof of \thmref{Maximal-Principle}]
Let $\calS$ be the speed assignment defined in \eqnref{Eq:Speed}, let $t_1>0$ be a time 
for which the folding with the speed $\calS$ is defined (see  \propref{Prop:Fold}) and let 
$\alpha$ be any loop. Denote the geodesic representative of $\alpha$ in $x$
with $\alpha_x$ and in $\by$ with $\alpha_\by$. Then $\alpha_x$ can be
sub-divided to segments $u_1 \cup w_1 \cup ... \cup u_m \cup w_m$ so that,
for $i=1, \dots, m$, 
\begin{itemize}
\item The segments $u_i$ are vanishing paths in $x$.  
\item The segments $\Phi(w_i)$ are immersed and $\alpha_\by = \cup_i \Phi(w_i)$. 
\end{itemize}
In particular, 
\[
|\alpha|_x = |\alpha|_\by + \sum_i |v_i|_x.
\] 
Let $v_i$ be a lift of $u_i$ to $T_x$. We can assume $v_i$ are completely disjoint 
from each other. Since $w_i$ are all legal, there is a one-to-one correspondence 
between sub-gates in $\T_\alpha$ and in $\cup_i \Theta_{v_i}$. 
Thus, \lemref{Vanishing-Path} implies
\begin{equation}\label{vanishinginequality}
\sum_i |u_i|_x = \sum_i |v_i|_x \leq 2 \sum_i \sum_{\sigma \in \Theta_{v_i}} s_\sigma
= 2 \sum_{\tau \in \T_\alpha} s_\tau.
\end{equation}

We have
\begin{align*}
\frac{|\alpha|_x - \sum_i |u_i|_x}{|\by|} 
  &= |\alpha|_y = |\alpha|_x + ( |\alpha|_y - |\alpha|_x)\\
|\alpha|_x - \sum_i |u_i|_x 
  & = |\alpha|_x |\by| + ( |\alpha|_y - |\alpha|_x) |\by|. \\
\intertext{By \eqnref{vanishinginequality}, we replacing $\sum_i |u_i|_x$ with $2 \sum_{\tau \in \T_\alpha} s_\tau$, we get} 
|\alpha|_x (1 - |\by|)  - 2 \sum_{\tau \in \T_\alpha} s_\tau 
 & \leq ( |\alpha|_y - |\alpha|_x) |\by|.\\
\intertext{But $(1 - |\by|) =\sum_{\tau \in \T} \ell_\tau \geq |\calS|$, therefore}   
|\alpha|_x - 2  \sum_{\tau \in \T_\alpha} \frac{s_\tau}{|\calS|} 
    & \leq ( |\alpha|_y - |\alpha|_x) \frac{ |\by|}{1- |\by|}. 
\end{align*} 
Note that
\[
d(x,y) = \log \frac 1{|\by|}  
\ \Longrightarrow\
 |\by| = e^{-d(x,y)}
\ \Longrightarrow\ 
\frac{ |\by|}{1- |\by|} = \frac 1{e^{d(x,y)} -1}. 
\]
Hence, letting $V_\alpha=2 \sum_{\tau \in \T_\phi(\alpha)} \frac{s_\tau}{|\calS|}$, 
we have 
\begin{equation} \label{Eq:V}
|\alpha|_x -V_\alpha \leq \frac{ |\alpha|_y - |\alpha|_x}{ e^{d(x,y)} -1}. 
\end{equation}

Now re-parametrize the folding path with arc-length and denote the
new parameter with $s$. Solving the 
differential equation given in \lemref{Length-Derivative}, we have
\begin{equation}
\dot{|\alpha|}_s = |\alpha|_s - V_\alpha 
\quad\Longrightarrow\quad
|\alpha|_s = (|\alpha|_x-V_\alpha)e^s + V_\alpha. 
\end{equation}
Note that, if $|\alpha|_y \leq |\alpha|_x$ then by \eqnref{Eq:V} the derivative of
the length is negative and $|\alpha|_t \leq |\alpha|_x$.
If $|\alpha|_y \geq |\alpha_x$ then,
\begin{align*}\label{modify}
|\alpha|_s &= (|\alpha|_x-V_\alpha)e^s + V_\alpha\\
  & =  (|\alpha|_x-V_\alpha)(e^s-1) + |\alpha|_x\\
\tag{\eqnref{Eq:V}}  
&\leq   ( |\alpha|_y - |\alpha|_x) \frac{e^s-1}{e^{d(x,y)} -1} + |\alpha|_x\\
\tag{$s \leq d(x,y)$} & \leq |\alpha|_y.
\end{align*}
That is, in either case, $|\alpha|_t \leq \max\big( |\alpha_x, |\alpha|_y\big)$. 
This finishes the proof. 
\end{proof}

\section{decorated difference of markings map} \label{Sec:Decorated} 
In \secref{weaklyconvex}, we constructed a balanced folding path starting from $x$ 
towards $y$ 
assuming that there is a difference of markings map $\phi \from x \to y$ such that 
$x_{\phi} = x$. In general, such a difference of marking map does not exit. In this section,
we modify the difference of markings map such that, for any two points $x, y \in \CV$, one 
can still define a notion of a folding path from $x$ to $y$. Such a modified difference of 
markings map will be called \emph{a decorated difference of markings map}.
Recall that, given $\phi \from x \to y$, one can construct a standard geodesics where 
first every edge $e$ outside of $x_{\phi}$ is shortened. For a decorated difference of 
markings map, we instead create an illegal turn in the interior of the edge $e$. 
The folding of that illegal turn effectively shortens the length of $e$. 

\subsection{Decorating the graphs}
Consider a pair of points $x, y \in \CV$ and an optimal map $\phi \from x \to y$.
Recall that $x_{\phi}$ is the tension subgraph of the map $\phi$, i.e. the subgraph
of $x$ consisting of edges that are stretched by a factor $\lambda$, where
$\log \lambda=d(x,y)$. We start by decorate the graph $x$(See \figref{decorate} for the decoration described here). 
 
Let $e$ be an edge outside of $x_{\phi}$, say connecting $v_0$ to $v_3$.  
Add two subdividing vertices $v_1$ and $v_2$ to $e$ so that the following holds:
Let $e_{i,j}$ denote the edge with end vertices $v_i$ and $v_j$. Then, 
we require that 
\[
|e_{0,1}|_x = |e_{1,2}|_x, \qquad\text{and}\qquad
\lambda \, |e_{2,3}|_x = |\phi(e)|_y. 
\]
This is always possible since $\lambda |e|_x < |\phi(e)|_y$. 
We call $v_1$ and $v_2$ \emph{pseudo-vertices} and refer to the graph $x$ 
with all the pseudo-vertices added as $x^d$.

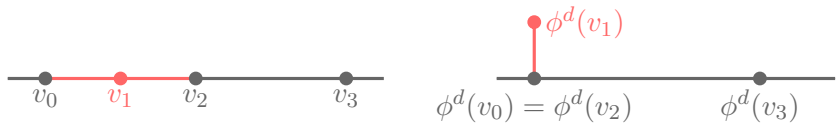
\begin{figure}[h!]
\begin{tikzpicture}

\filldraw [color=red!60, very thick]
(1,0) -- (2,0) circle (2pt) node[align=center, below] {$v_1$} -- (3,0) ;

\filldraw [color=black!60, very thick] 
 (.5,0) -- (1,0) circle (2pt) node[align=left,   below] {$v_0$} ;

 \filldraw [color=black!60, very thick]
(3,0) circle (2pt) node[align=center, below] {$v_2$} --
(5,0) circle (2pt) node[align=right,  below] {$v_3$} -- (5.5,0);



\filldraw [color=red!60, very thick]
(7.5,0) -- (7.5, 0.75) circle (2pt) node[align=center, right] {$\phi^d(v_1)$};

 \filldraw [color=black!60, very thick]
(7,0) -- (7.5,0) circle (2pt) node[align=center, below] {$\phi^d(v_0)=\phi^d(v_2)$} --
(10.5,0) circle (2pt) node[align=right,  below] {$\phi^d(v_3)$} -- (11.5,0);

\end{tikzpicture}
\caption{Decoration of edges in $x$ and $y$.}
\label{decorate}
\end{figure} 

Next we decorate $y$.  For every $e$ and vertex $v_0 \in x$ as above, we attach 
a new edge of length  $\lambda |e_{0,1}|_x$ to $y$ at the point $\phi(v_0) \in y$. 
The other end 
of this new edge is incident to an added degree-one vertex. Thus we have added a leaf at 
$\phi(v_0)$. 
We label this edge, oriented away from $\phi(v_0) \in y$, by $\epsilon_e$, and the same edge with 
opposite orientation by $\epsilon^{-1}_e$. 
We do this for every edge in $x$ that is not in $x_\phi$ and denote the resulting decorated 
graph by $y^d$. 

We now modify the optimal map $\phi \from x \to y$ to $\phi^d \from x^d \to y^d$.
Define $\phi^d$ to be a map that is linear on every edge with slope $\lambda$. 
For edges in $x_{\phi}$ the two maps agree. For every edge $e$ 
outside of $x_{\phi}$, we have 
\begin{align*}
\phi^d(e_{0,1}) &= \epsilon_e\\
\phi^d(e_{1,2}) &= \epsilon^{-1}_e\\
\phi^d(e_{2,3}) &= \phi(e)
\end{align*}
Note that, the vertex $v_1$ is a one-gate vertex. 
 
\subsection{Folding paths under decorated difference of markings maps}
We can use the decorated graphs to construct a folding path from $x$ to $y$.  The graphs
$x^d, y^d$ are both constructed from $x, y$ by adding vertices and hairs, respectively. 
The graph  $y$ is the \emph{core graph} of $y^d$. The graphs and 
maps $\phi^d: x^d \to y^d$ define a train track structure on $x^d$ that is slightly 
different from its definition in Section 2: in $x^d$, each point either has two gates or 
is mapped to the end of a leaf. When the context is clear we also write the decorated difference of marking map as $\phi^d \from x \to y$.

Let $\calT$ be the set of sub-gates of $\phi^d$.  As done previously, we can fold 
$x^d$ according to a given speed assignments $\calS= \{s_\tau\}_{|\tau|=2}$
by identifying edges associated to $\tau$ along a sub-segment of length $t\, s_\tau$
for $t$ small enough, to obtain a quotient map $\bphi_t^d \from x^d \to \overline x_t^d$. 
The graph is a union of a core graph $\bx_t \in cv$ and a hair of size $t \, s_\tau$
associated to each $\tau$. Let $x_t$ be the point in $\CV$ that is a normalization 
$\bx_t$ to a graph of total length 1 and let $x^d_t$ be the scaling of $\bx^d_t$
by the same factor. Let $\phi^d_t \from x^d \to x^d_t$ be the composition 
of $\bphi^d_t$ and scaling and let $L\big(\phi^d_t\big)$ denote the Lipschitz constant of 
$\phi^d_t$. By definition of distance,  $\log L\big(\phi^d_t\big)\leq d(x,x_t)$. 
But also, $L\big(\phi^d_t\big)$ is the scaling factor form $\bx_t$ to $x_t$
and the length of any loop that is legal in $x^d$ increases by this factor from 
$x$ to $x_t$. Thus, 
\[
\log L\big(\phi^d_t\big) = d(x,x_t). 
\]

As before, for $t$ small enough, the left over map 
\[
\psi^d_t \from x^d_t \to y^d, \qquad
\psi^d_t = \phi^d \circ (\phi^d_t)^{-1}
\]
is defined because, for small $t$, points that are identified under $\phi_t^d$ are also 
identified under $\phi^d$. Also, since the Lipschitz constant of these maps 
are constant, we have 
\[
L\left(\psi^d_t\right) = \frac{L\big(\phi^d\big)}{L\big(\phi^d_t\big)}
\quad \Longrightarrow\quad 
d(x_t, y) \leq \log L\left(\psi^d_t\right) \leq  d(x,y) - d(x, x_t). 
\]
But $d(x,x_t) + d(x_t, y) \leq d(x, y)$ by the triangle inequality. Therefore,  
\[
d(x,x_t) + d(x_t, y)= d(x, y)
\]
and hence the path $\{x_t\}$ is a geodesic starting from $x$ towards $y$. 
To summarize, similar to \propref{Prop:Fold} we have 

\begin{proposition}\label{Prop:Decorated-Folding}
Given any two points $x$, $y$ in $\CV$ there exists a decorated difference of markings map 
$\phi^d \from x \to y$ such that $x^d_{\phi^d} = x^d$. Furthermore, any speed
assignment $\calS$ defines a geodesic $\gamma \from [0,t_1] \to CV$ starting from 
$x$ towards $y$, for some $t_1>0$. 
\end{proposition}

\begin{remark} \label{Rem:Decorated} 
The decorated difference of markings map is non-unique in several ways. First, it depends
on the difference of markings map $\phi$ which is not unique. Also, the pseudo-vertices 
can be added at either end or in the middle of an edge. That is, there is a family of 
decorated folding paths connecting any pair of points. 
\end{remark} 

We now prove an analogue of \thmref{Maximal-Principle} for decorated folding paths 
following closely the constructions and arguments of \secref{weaklyconvex}. We define 
the length loss function $\ell_{\tau}$ at each illegal turns based on the decorated 
difference of markings map from $x^d$ to $y^d$ as we do in \secref{weaklyconvex}. 
The associated function $s_{\tau} $ is also identically defined as in \eqnref{Eq:Speed}
and it defines a geodesic from $x$ towards $y$ as in \propref{Prop:Decorated-Folding}. 
We keep track of which illegal turns are from the original graph $x$ and which are at the
pseudo-vertices. Let  $\T^C$ denotes the set of all sub-gates in the undecorated 
graph $x$ (as a subset of $x^d$), let $\T^H$ denotes the set of sub-gates at the 
pseudo-vertices of $x^d$ and let $\T$ be the union of $\T^C$ and $\T^H$. 
We always assume $t \in [0,t_1]$ (\propref{Prop:Decorated-Folding}), in particular,
$x_t$ is in the same simplex at $x$. From $x^d$ to $\bx^d_t$, after time $t$, 
we define analogously:
\begin{equation}\label{definitionofSd}
|\calS|^d = \frac{1-|\bx^d_t|}{t}.
\end{equation}
However, our actual target graph is the core graph $\bx_t$ of $\bx^d_t$. 
Observe that  the length of $\bx_t$ is the length of $\bx^d_t$ minus the length of the set of 
hairs. The length of each hair associated to a sub-gate $\tau$ is $t s_{\tau}$ by 
construction, thus
\begin{equation}\label{difference}
|\bx_t| = |\bx^d_t| - t\sum_{\tau' \in \T^H} s_{\tau'}=|\bx^d_t| - t\sum_{\tau' \in \T^H} \ell_{\tau'}.
\end{equation}
In the second equality, the assertion $s_{\tau'} =\ell_{\tau'}$ comes from the fact that at 
each pseudo-vertex there is only one sub-gate that contains the only illegal turn $\tau$.
We also define the length loss speed $|\calS|$ with respect to $\bx_t$:
\begin{equation}\label{newS}
|\calS| = \frac{1-|\bx_t|}{t}.
\end{equation}
Substitute \eqnref{difference} into \eqnref{newS}, we have 
\begin{equation}\label{newlemma3.5}
|\calS| = |\calS|^d +\sum_{\tau' \in \T^H} \ell_{\tau'}.
\end{equation}
By \lemref{lengthlossfunction} and \lemref{lengthlossspeed} we have 
\[
1-|\by^d| = \sum _{\tau \in \T} \ell_{\tau} > |\calS|^d.
\]
By \eqnref{newlemma3.5} that becomes
\begin{equation}\label{new3.5}
1-|\by^d| = \sum _{\tau \in \T} \ell_{\tau} +\sum_{\tau' \in \T^H} \ell_{\tau'}> |\calS|.
\end{equation}
Now we are ready to prove the main theorem of this section:

\begin{theorem} \label{Thm:Decorated-MP}
Given a difference of markings map $\phi^d \from x^d \to y^d$, $x, y \in \CV$, where 
$x^d_\phi=x^d$, there exists a speed assignment $\calS$ defining a folding path 
$\gamma \from [0,t_1] \to \CV$ starting at $x$ towards $y$ so that, 
for every loop $\alpha$ and every time $t \in [0,t_1]$, 
\[
|\alpha|_t \leq \max \big(  |\alpha|_x, |\alpha|_y \big).
\]
\end{theorem}

\begin{proof}
We follow the proof of \thmref{Maximal-Principle}. Let $\calS$ be the speed
assignment describe above and $\gamma$ be the associated geodesic 
starting from $x$ towards $y$ coming from \propref{Prop:Decorated-Folding}. 
Recall that a decorated difference of markings map fold $x$ into a graph with hairs 
$\bx^d_t$ and where the core graph of $\bar{x}^d_t$ is $\bx_t$. By Lemma 3.2,  
for a given loop $\alpha$,  
\[
\dot{|\alpha|}_{t=0}
 = |\alpha|_x - 2 \sum_{\tau \in \T_\phi(\alpha)} \frac{s_\tau}{|\calS|}.
\]
By \lemref{lengthlossfunction}, 
$$
\sum_{\tau \in \T} \ell_{\tau} = |x| -|\bar{x}|.
$$
However, since our final target is $\by$, we separate the sub-gates $\tau$ in the core graph and the sub-gates $\tau'$ at
the pseudo-vertices. Let  $\T^C$ denotes the set of all sub-gates in the core graph and $\T^H$ denotes the set of sub-gates
at the pseudo-vertices; $\T$ is a union of $\T^C$ and $\T^H$. We have $|\bar{x}| -|\by| = \sum_{\tau' \in \T^H} \ell_{\tau'}$.
Therefore we can rewrite \lemref{lengthlossfunction} as
$$
\sum_{\tau \in \T} \ell_{\tau} + \sum_{\tau' \in \T^H} \ell_{\tau'}= \sum_{\tau \in \T^C} \ell_{\tau} + 2\sum_{\tau' \in \T^H} \ell_{\tau'}= |x| -|\by|.
$$
By \eqnref{new3.5}, instead of \lemref{lengthlossspeed}, we have

$$\
sum_{\tau \in \T} \ell_{\tau} + \sum_{\tau' \in \T^H} \ell_{\tau'} = \sum_{\tau \in \T^C} \ell_\tau + 2\sum_{\tau' \in \T^H} \ell_{\tau'} \geq |\calS|.
$$

Therefore, for every loop $\alpha$ such that $|\alpha|_y \leq |\alpha|_x$, 
\begin{align*}
 |\alpha|_y= \frac{|\alpha|_x - \sum_i |u_i|_x}{|\by|}  &\leq |\alpha|_x\\
 |\alpha|_x - \sum_i |u_i|_x  &\leq |\alpha|_x \, |\by| 
     = |\alpha|_x \left(1 - \sum_{\tau \in \T^C} \ell_\tau - 2\sum_{\tau' \in \T^H} \ell_{\tau'} \right)\\
 |\alpha_x| \left( \sum_{\tau \in \T^C} \ell_\tau +2 \sum_{\tau' \in \T^H} \ell_{\tau'} \right)
   &\leq \sum_i |u_i|_x \leq 2 \sum_{\tau \in \T_\alpha} s_\tau.\\
\intertext{But $| \calS | \leq  \sum_{\tau \in \T^C} \ell_\tau + 2\sum_{\tau' \in \T^H} \ell_{\tau'} $, therefore}   
|\alpha|_x &\leq 2  \sum_{\tau \in \T_\alpha} \frac{s_\tau}{|\calS|}.
\end{align*}
It follows that $|\alpha|_x < 2 \sum_{\tau \in \T_{\alpha}} \frac{s_{\tau}}{| \calS|}$ and 
$\dot{|\alpha|}_{t=0} \leq 0$.
\end{proof}

\section{Construction of the balanced folding path} 
In this section we prove \thmref{Thm:Length} restated below:

\begin{theorem} 
\label{Length} 
Given points $x, y \in \CV$, there exists a geodesic $[x,y]_{\rm bf}$ from $x$ to $y$ 
so that, for every loop $\alpha$, and every time $t$, 
\[
|\alpha|_t \leq \max \big(  |\alpha|_x, |\alpha|_y \big).
\]
\end{theorem}

\begin{proof}
Given $x$ and $y$, we consider the decorated graphs $x^d$ and $y^d$ and the
decorated difference of markings map $\phi^d \from x^d \to y^d$. Applying 
\thmref{Thm:Decorated-MP}, we obtain a geodesic $\gamma \from [0,t_1] \to \CV$
starting from $x$ towards $y$. Now we consider the pair of points $x_{t_1}$ and $y$
and apply \thmref{Thm:Decorated-MP} again to continue the geodesic to
an interval $[t_1, t_2]$. Continuing in this way, we either reach $y$ after
finitely many steps or limit to a point $x' \in \CV$. Note that every point $x_t$
along this path has the property that 
\begin{equation} \label{Eq:On-Path}
d(x, x_t) + d(x_t,y) = d(x,y)
\end{equation} 
and the set of such points is a compact subset of $\CV$. Hence, the same holds for $x'$. 
In particular $x'$ is a point in $\CV$ and the geodesic does not exit $\CV$. 

Now, we can apply \thmref{Thm:Decorated-MP} to the pair $x'$ and $y$ and continue 
the geodesic even further, getting closer to $y$. This results in a geodesic 
connecting $x$ to $y$ because if the process stops at some points $x''$ before $y$, 
then $x''$ is still in the compact set defined by 
\eqnref{Eq:On-Path} and we could apply \thmref{Thm:Decorated-MP} again
to go further. 

We re-parametrize this geodesic by arc-length to obtain 
$\bgamma \from [0,d] \to \CV$, $d=d(x,y)$ (and use the parameter $s$ to emphasize
this fact). Let $\Sigma \subset [0,d]$ be the closure of set of times, each of which is an endpoint 
of an interval coming from an application of \thmref{Thm:Decorated-MP}.
Note that, if $s \in \Sigma$, then an interval to the right of $s$ is not in $\Sigma$. 
Hence, $\Sigma$ is a well ordered set. That is, $[0,d]$ is a union of the interiors of
countably many intervals coming from \thmref{Thm:Decorated-MP} and 
the countable, well ordered set $\Sigma$ which includes $0$ and $d$. 

For any loop $\alpha$, we prove the theorem using transfinite induction on $\Sigma$. 
That is, for every time $s \in \Sigma$, we show
\begin{equation} \label{Eq:Main}
|\alpha|_s \leq \max\big(|\alpha|_x, |\alpha|_y \big). 
\end{equation}
The theorem for other times then follows from \thmref{Thm:Decorated-MP}. 

\eqnref{Eq:Main} is clearly true for $s=0$. For any $s \in \Sigma$, assume 
\eqnref{Eq:Main} holds for every $s' \in \Sigma$ with $s' <s$. We need to show that it 
also holds for $s$. There are two cases. If $s$ is an end point of an interval $[s',s]$ 
coming from \thmref{Thm:Decorated-MP}, then, by \thmref{Thm:Decorated-MP}
\[
|\alpha|_s \leq \max\big(|\alpha|_{s'}, |\alpha|_y \big)
\]
and by the assumption of induction 
\[
|\alpha|_{s'} \leq \max\big(|\alpha|_x, |\alpha|_y \big)
\]
and the conclusion follows. 

Otherwise, there is a sequence $s_i \in \Sigma$, with $s_i<s$, so that $s_i \to s$. 
By the assumption of induction, we have
\[
|\alpha|_{s_i} \leq \max\big(|\alpha|_x, |\alpha|_y \big). 
\]
But the length of $\alpha$ is a continuous function over $\CV$. Taking a limit, we obtain the theorem. 
\end{proof}

\section{non-convexity}\label{counterexample}
In this section we present examples that combine to prove Theorem ~\ref{Thm:Negative}. 
Some of the examples are done in low rank free groups, however, they can clearly be 
generalized to higher rank. First we show that there are points in Outer space such that 
no geodesic between them gives rise to convex length functions for all curves.

\begin{proposition} \label{Ex:Non-Convex} 
There are points $x,y \in \CV$ and a loop $\alpha$ so that along any geodesic 
connecting $x$ to $y$, the length of $\alpha$ is not a convex function of distance in $\CV$. 
\end{proposition}

\begin{proof}
We construct a simple example in ${\rm CV}_2$. Let $a$ and $b$ be generators for
$\F_2$. Let $x \in {\rm CV}_2$ be a graph that consists of two simple loops labeled $a$ 
and $b$, wedged at a vertex $v$ where each loop has length $\frac 12$ (a rose with
two pedals). Let $\by$ be a 
quotient of $x$ obtained by identifying two subsegments of length $\frac 18$ in 
the loop labeled $a$. Then $\by$ is a rank $2$ graph in the shape of a dumbbell
with total length $\frac 78$, where the $b$--loop has length $\frac 12$ and the $a$--loop 
has a length $\frac 14$. Let $y$ be $\by$ rescaled to have length $1$ (by a factor $\frac 87$). 
There is a rigid folding path from $x$ to $y$ because $\by$ was obtained from $x$ by
identifying two sub-edges. Hence, this path  $[x,y]_{\rm f}$  is the unique geodesic 
connecting $x$ to $y$ (see \remref{Rem:Rank-2}). Let $\alpha$ be the loop representing
element $a \in \F_2$. Then
\[
|\alpha|_x = \frac 12 \qquad\text{and}\qquad |\alpha|_y = \frac 87 \cdot \frac 14 = \frac 27. 
\]

Consider the length of $|\alpha|_t$ of $\alpha$ along this folding path. 
By \lemref{Length-Derivative}, the derivative of the length of $\alpha$ at $x$ is:
$$
\dot{|\alpha|_t} \Big|_{t=0} = |\alpha|_x -2 = \frac 12 -2 < 0. 
$$
And since the length of $\alpha$ is decreasing the derivative stays negative. 
In fact, for $s>0$
$$
\dot{|\alpha|_t} \Big|_{t=s} = |\alpha|_s -2 
$$
is a decreasing function as well. Therefore $|\alpha|_t$ is concave along this folding path. 
That is, there is no geodesic between $x$ and $y$ on which the length of $\alpha$ is 
a convex function of distance. 
\end{proof}

We now examine if \thmref{Thm:Length} hold for other geodesics connecting two points
in $\CV$. We start by looking at a general folding path and we show that a folding path
with end points in a small ball can still go arbitrarily far away from the center of the ball.  

\begin{proposition} \label{Ex:General} 
For any $R >0$, there are points 
$x,y,z \in \CV$ and there is a folding path $[y,z]_{\rm ng}$ connecting 
$y$ to $z$ so that 
\[
  y,z \in \Bo(x,2)
  \qquad\text{and}\qquad
  [y,z]_{\rm ng} \not \subset \Bo(x,R).
\] 
\end{proposition} 

\begin{proof}
We construct an example in ${\rm CV}_3$. For higher rank Outer spaces,
one can modify the example to roses with more loops such that the optimal map outside 
of the simple loops labeled $a, b, c$ is identity.

Consider constants $\ep>0$, $\delta>0$ and an integer $m>0$ so that 
\[
\ep \leq \delta \ll 1 \qquad\text{and}\qquad
m \, \delta = 1-3\delta.
\] 
Assume $\F_3$ is generated by elements $a$, $b$ and $c$ and let $x, y, z$ and $w$ 
be points in ${\rm CV}_3$ which are wedges of simple loops with lengths and labels 
summed up in the table below. 
 
 \begin{center} \begin{tabular}{c|cc|cc|cc|cc|}
\cline{2-9}
& \multicolumn{2}{ c| }{$x$}& \multicolumn{2}{ c| }{$y$}& \multicolumn{2}{ c| }{$w$}& \multicolumn{2}{ c| }{$z$}\\ 
\cline{2-9} 
& \text{label} & length & label & length& label & length& label & length \\
\cline{1-9} 
\multicolumn{1}{ |c| }{Edge 1 $\rule[-4mm]{0mm}{10 mm}$ } 
    & $a$ & $ \epsilon$ & $ab$ & $\delta + \delta^2$  
    & $ab$ & $\frac {1+\delta}3$ & $a$ & $\frac \delta2$  \\ 
\cline{1-9}
\multicolumn{1}{ |c| }{Edge 2  $\rule[-4mm]{0mm}{10 mm}$} 
    & $b$ & $\frac 12$ & $b$ &  $\delta$  
    & $b$ & $\frac 13$ & $b$ & $\frac 12$ \\ 
\cline{1-9}
\multicolumn{1}{ |c| }{Edge 3  $\rule[-4mm]{0mm}{10 mm}$} 
    & $c$ & $\frac {1-\epsilon}2$  & $cb^m$ & $1-2\delta - \delta^2$  
    & $c$ & $\frac {1-\delta}3$ & $c$ & $\frac {1-\delta}2$ \\ 
\cline{1-9}
\end{tabular}
\end{center}

Note that if, in $y$, we fold the edge labeled $cb^m$ $m$--times around 
$b$ (without rescaling), we obtain a graph $\overline{w}$ with labels $ab$, $b$ and $c$ 
and lengths $(\delta+\delta^2)$, $\delta$ and 
\[
(1-2\delta - \delta^2) - m \, \delta = (1-2\delta - \delta^2) - (1-3\delta) = 
\delta - \delta^2
\]
which is a graph that is projectively equivalent to $w$ (by a factor of $\frac 1{3\delta}$). 
Similarly, if in $w$, we fold the edge labeled $ab$ once around $b$ (without rescaling), 
we obtain a graph $\overline{z}$ with labels $a$, $b$ and $c$ and
lengths $\frac \delta 3$, $\frac 13$ and $\frac {1-\delta}3$ which a graph 
that is projectively equivalent to $z$ by a factor $\frac 32$. Therefore, 
there is a folding path from $y$ to $z$ that passes through $w$. But
this is not a greedy folding path since the edge labeled $ab$ is not
folded around $b$ in the segment $[y,w]$. 

Let $\alpha$ be the loop representing the element $a \in \F_3$. Then 
\begin{align*}
d(x,y) &=  \frac { |\alpha|_y }{ |\alpha|_x } = \log \frac { \delta + \delta^2}{ \epsilon }  \\
d(x,z) &=  \frac { |\alpha|_z }{ |\alpha|_x } = \log \frac { \delta/2 }{ \epsilon }  \\
d(x,w) &=  \frac { |\alpha|_w }{ |\alpha|_x } = \log \frac { (1+\delta)/3 }{ \epsilon }  
   \geq \log \frac 1{3\ep}. 
\end{align*}
If, for example, we let $\ep =\frac {\delta}{5}$ then $y, z \in \Bo(x, 2)$, but $w$ 
can be made arbitrarily far away by making $\delta$ small. 
\end{proof}

Next we consider standard geodesic paths connecting two points which are the type of geodesics 
most often considered to connect two arbitrary points in $\CV$ (not every pair of points
can be connected via a folding path). The situation is improved
somewhat but, by taking the ball large enough, one can construct examples
where a standard geodesic that has its endpoints in a ball goes arbitrarily far from the ball. 

\begin{proposition}\label{Ex:Standard} 
There exists a constant $c>0$ such that,
for every $R>0$, there are points $x,y,z \in \CV$ and a standard geodesic $[y,z]_{\rm std}$ 
connecting $y$ to $z$ such that 
\[
  y,z \in \Bo(x,R)
  \qquad\text{and}\qquad
  [y,z]_{\rm std} \not \subset \Bo(x,2R-c).
\]   
That is, the standard geodesic path can travel nearly twice as far from $x$ as $y$ and $z$ 
are from $x$. 
\end{proposition}

\begin{proof}
As before, we construct the example in ${\rm CV}_3$. Let $\psi \in {\rm Out}(\F_3)$ 
be defined as follows
\begin{align*}
\psi(a) &=  ab  &  \psi^{-1}(a) & = b \\
\psi(b) & = a    &  \psi^{-1}(b) &= b^{-1}a\\
\psi(c) & = c    &   \psi^{-1}(c)&= c 
\end{align*}
It is known (and easy to see) that, for any integer $m>0$, the word length of $\psi^m(a)$ 
is $F_{m+3}$ and the word length of $\psi^m(b)$ is $F_{m+2}$, where $F_i$
is the $i$-th Fibonacci number. Similarly, the word length of $\psi^{-m}(a)$ is 
$F_{m+2}$ and the word length of $\psi^{-m}(b)$ is $F_{m+3}$. 
For a large integer $m>0$, let 
\[
\delta = \frac{1}{F_{m+2} + F_{m+3} + 1}
\]
and consider points  $x, y, z, w \in {\rm CV}_3$ which are wedges 
of simple loops and where the lengths and edge labels are summed up in the table. 
  
\begin{center} \begin{tabular}{c|cc|cc|cc|cc|}
\cline{2-9}
& \multicolumn{2}{ c| }{$x$}& \multicolumn{2}{ c| }{$y$}& \multicolumn{2}{ c| }{$w$}& \multicolumn{2}{ c| }{$z$}\\ 
\cline{2-9} 
& \text{label} & length & label & length& label & length& label & length \\
\cline{1-9} 
\multicolumn{1}{ |c| }{Edge 1 $\rule[-4mm]{0mm}{10 mm}$ } 
    & $a$ & $ \delta$ & $\psi^m (a)$ & $ \delta$  & $\psi^m (a)$ & $F_{m+3} \, \delta$ & $a$ & $\frac 13$  \\ 
\cline{1-9}
\multicolumn{1}{ |c| }{Edge 2  $\rule[-4mm]{0mm}{10 mm}$} 
    & $b$ & $\delta$ & $\psi^m(b)$ &  $ \delta$  & $\psi^m(b)$ & $F_{m+2} \, \delta$ & $b$ & $\frac 13$  \\ 
\cline{1-9}
\multicolumn{1}{ |c| }{Edge 3  $\rule[-4mm]{0mm}{10 mm}$} 
    & $c$ & $1-2\delta$  & $c$ & $1-2\delta$ & $c$ & $\delta$  & $c$ & $\frac 13$ \\ 
\cline{1-9}
\end{tabular}
\end{center}

If we let $\overline z$ be the rose with labels $a$, $b$ and $c$ and all edge lengths 
$\delta$, then there is a quotient map $\bphi \from w \to \overline{z}$ that maps the edge of 
$w$ labeled $\psi^m(a)$ to an edge path containing $F_{m+3}$ edges and maps the edge 
of $w$ labeled $\psi^m(b)$ to an edge path containing $F_{m+2}$ edges. The graph $z$
is obtained from $\overline{z}$ by scaling by a factor $\frac 1{3\,\delta}$. Hence, 
$w$ can be connected to $z$ using a folding path and tension graph of 
$\phi \from w \to z$ is all of $w$. The map from $y$ to $w$ is scaling two of the edges
and contracting the third. Hence, the standard geodesic from $y$ to $z$ passes through 
$w$. 

Next, we compute the distance from $x$ to these points. Let $\alpha$ be the loop 
representing the element $a \in \F_3$ and $\beta$ be the loop representing $b \in \F_3$. 
The loop $\alpha$ has a combinatorial length $F_{m+2}$ (which is the word length of 
$\phi^{-m}(a)$) in both $y$ and $w$ and $\beta$ has a combinatorial length $F_{m+3}$ 
(which is the word length of $\phi^{-m}(b)$) in both $y$ and $w$. In particular, 
\[
|\beta|_w \geq F_{m+3} \cdot (F_{m+2}\,\delta)
\]
because the geodesic representative of $\beta$ in $w$ consists of $F_{m+3}$ edges
each having a length of at least $F_{m+2}\,\delta$. We have
\begin{align*}
 d(x,y) & =  \log \frac{|\beta|_y}{|\beta|_x} =  \log \frac{F_{m+3}\, \delta}{\delta} = \log F_{m+3}\\
 d(x,z) & =  \log \frac{|\alpha|_z}{|\alpha|_x} = \log \frac{|\beta|_y}{|\beta|_x} 
     = \log \frac{1/3}{\delta}= \log \frac{1}{3 \delta}\\
 d(x,w) & \geq  \log \frac{|\beta|_y}{|\beta|_x}
    > \log \frac{F_{m+3} \cdot (F_{m+2}\, \delta)} {\delta} = \log (F_{m+3} \, F_{m+2}). 
\end{align*}
We now set $R = \log F_{m+3}$ which is larger than $\log \frac 1{3\delta}$. 
Then, $y,z \in \Bo(x,R)$. There is a constant $c$, (slightly larger than the golden ratio) 
so that
\[
 \log (F_{m+2} \, F_{m+3}) \geq 2  \log (F_{m+3}) - c = 2R-c
\]
which implies $w \not \in \Bo(x, 2R-c)$. This finishes the proof. 
\end{proof}

The most well-behaved geodesic often considered is a greedy folding path. 
In fact, as mentioned in the introduction, the lengths of curves are quasi-convex
function of distance along a greedy folding path. However, we show that a greedy folding 
path with end point inside of a ball may exit the ball. 

\begin{proposition}
For $n\geq 4$ and every $R>0$, there are points $x,y,z \in {\rm CV}_{n+2}$ where $y$ and 
$z$ are connected by a greedy folding path $[y,z]_{\rm gf}$ such that 
\[
y,z \in \Bo(x, R) \qquad\text{but}\qquad [y,z]_{\rm gf} \not \subset \Bo(x, R). 
\]
\end{proposition}
\begin{proof}

Let $x, y, z, w \in {\rm CV}_{n+2}$ be four graphs that are each a bouquets of $n+2$ simple 
loops. Consider $\F_{n+2}$ as being generated by $a$, $b$ and $c_i$, for $i=1, \dots, n$. 
The lengths and the labels of these graphs are described in the table below where
$\ep$ is a small positive number. 

\begin{center} \begin{tabular}{c|cc|cc|cc|cc|}
\cline{2-9}
& \multicolumn{2}{ c| }{$x$}& \multicolumn{2}{ c| }{$y$}& \multicolumn{2}{ c| }{$z$}& \multicolumn{2}{ c| }{$w$}\\ 
\cline{2-9} 
& \text{label} & length & label & length& label & length & label & length\\
\cline{1-9} 
\multicolumn{1}{ |c| }{Edge 1 $\rule[-4mm]{0mm}{10 mm}$ } 
   & $a$ & $\ep$  & $a b^2 $ & $3/(2n+4)$ & $a$ & $1/(n+2)$  & $ab$ & $2/(n+3)$\\ 
\cline{1-9}
\multicolumn{1}{ |c| }{Edge 2  $\rule[-4mm]{0mm}{10 mm}$} 
   & $b$ & $\frac{(1-\ep)}{2}$ & $b$ & $1/(2n+4)$   & $b$ & $1/(n+2)$  & $b$ & $1/(n+3)$\\ 
\cline{1-9}
\multicolumn{1}{ |c| }{Edge $i$  $\rule[-4mm]{0mm}{10 mm}$} 
    & $c_i$ & $\frac{(1-\ep)}{2n}$  & $c_i b$ & $2/(2n+4)$  & $c_i$ & $1/(n+2)$ & $c_i$ & $1/(n+3)$\\ 
\cline{1-9}
\end{tabular}
\end{center}

Note that, the greedy folding path from $y$ to $z$ passes through $w$. In fact, 
$[y,z]_{\rm gf}$ consists of two subsegments, in the first part $c_ib$ and $ab^2$ wrap 
around $b$ simultaneously to reach $w$, and in the second part the edge labeled $ab$
wraps around $b$ to reach $z$. The distance $d(y,w) = \log \frac{2n+4}{n+3}$ and 
the associated stretch factors of edges are
\begin{align*}
\lambda(ab^2) &=\frac{3/(n+3)}{3/(2n+4)} =\frac{2n+4}{n+3}\\
\lambda(b) &= \frac{1/(n+3)}{1/(2n+4)}=\frac{2n+4}{n+3}\\
\lambda(c_ib) &= \frac{2/(n+3)}{2/(2n + 4)}=\frac{2n+4}{n+3}\\
\end{align*}
are all the same. Likewise, the distance $d(w,z) = \log \frac{n+3}{n+2}$ and 
associated stretch factors of edges are
\begin{align*}
\lambda(ab) &= \frac{2/(n+2)}{2/(n+3)}= \frac{n+3}{n+2} \\
\lambda(b) &= \frac{1/(n+2)}{1/(n+3)}= \frac{n+3}{n+2}\\
\lambda(c_i) &=  \frac{1/(n+2)}{1/(n+3)} = \frac{n+3}{n+2}\\
\end{align*}
which again are the same for every edge. 

Next, we measure distances from the center of the ball $x$. Let $\alpha$ be the loop
associated to the element $a \in \F_3$. For $\ep$ small enough, all three distances are 
realized by the stretch factor associated to $\alpha$. That is, 
\begin{align*}
 d(x,y) & =  \log \frac{|\alpha|_y}{|\alpha|_x} = \log\frac{5/2n+4}{\epsilon} = \log \frac{5}{2n\epsilon + 4\epsilon}\\
 d(x,z) & =  \log \frac{|\alpha|_z}{|\alpha|_x} =  \log\frac{1/n+2}{\epsilon}= \log \frac{1}{n\epsilon + 2\epsilon}\\
 d(x,w) & =  \log \frac{|\alpha|_w}{|\alpha|_x} =  \log\frac{3/n+3}{\epsilon}= \log \frac{3}{n\epsilon + 3\epsilon}\\
\end{align*}
But, for all $n\geq 4$, we have 
\[
  \frac{3}{n\epsilon + 3\epsilon} > 
    \max \left( \frac{1}{n\epsilon + 2\epsilon},  \frac{5}{2n\epsilon + 4\epsilon} \right) 
\]
Thus, if we set $R=\frac{1}{n\epsilon + 2\epsilon}$, we have an example
of a greedy folding path with end point in $\Bo(x, R)$ that travels outside the ball. 
\end{proof}

\section{In-coming balls}\label{inball}

In contrast with out-going balls, we prove that in-coming balls are not weakly quasi-convex:

\begin{theorem}
For any constant $R>0$, there are points $x,y,z \in \CV$ such that, 
$y,z \in \Bi(x,2)$ but, for any geodesic $[y,z]$ connecting $y$ to $z$, 
\[
[y,z] \not \subset \Bi(x, R). 
\]
\end{theorem}

\begin{proof}
We show that there exists a family of balls
and pairs of points $y_m$ and $z_m$ in these balls such that the 
geodesic connecting $y_m$ to $z_m$ is unique and it travels arbitrarily far away 
from the center of the corresponding ball. Since the geodesic is unique, this can be
restated as: every geodesic connecting $y_m$ to $z_m$ travels arbitrarily far from the 
center of the balls.

Fix an integer $m>0$ and, as usual, let $a, b$ and $c$ be generators for $\F_3$. Examples in higher dimension can be adapted from this example by adding loops on which the map is identity along the path. Let $x=x_m$, $y=y_m$ and $z=z_m$ be roses with labels and lengths specified in the table.

\begin{center} \begin{tabular}{c|cc|cc|cc|cc|}
\cline{2-9}
& \multicolumn{2}{ c| }{$x$}& \multicolumn{2}{ c| }{$y$}& \multicolumn{2}{ c| }{$w$}& \multicolumn{2}{ c| }{$z$}\\ 
\cline{2-9} 
& \text{label} & length & label & length& label & length & label & length \\
\cline{1-9} 
\multicolumn{1}{ |c| }{Edge 1 $\rule[-4mm]{0mm}{10 mm}$ } 
    & $a$ & $ \frac 12 - \frac 1m$ & $ab^m$ & $\frac {m+1}{2m+4}$ & $a$ & $\frac {1}{m+4}$  & $a$ & $\frac 13$ \\ 
\cline{1-9}
\multicolumn{1}{ |c| }{Edge 2  $\rule[-4mm]{0mm}{10 mm}$} 
    & $b$ & $\frac 1m$ & $b$ & $\frac {1}{2m+4}$  & $b$ & $\frac {1}{m+4}$  & $b$ & $\frac 13$    \\ 
\cline{1-9}
\multicolumn{1}{ |c| }{Edge 3  $\rule[-4mm]{0mm}{10 mm}$} 
    & $c$ & $\frac 12$  & $cb^m a$ & $\frac {m+2}{2m+4}$ & $cb^m a$ & $\frac {m+2}{m+4}$ & $c$ & $\frac 13$ \\ 
\cline{1-9}
\end{tabular}
\end{center}

Note that, $w$ is obtained from $z$ by wrapping the edge labeled $ab^m$ around
the edge labeled $b$ $m$--times and then scaling by a factor $\frac {2m+4}{m+4}$. 
Throughout this portion, the illegal turn $\langle ab^m, b \rangle$ is the only
illegal turn. Similarly, $z$ is obtained from $w$  wrapping the edge labeled $ab^m$ 
around the edge labeled $a$ once, then around the edge labeled $b$ $m$--times
and finally scaling by $\frac{m+4}{3}$.  Again, during each sub-segment, there is exactly 
one non-yo-yo illegal turn; first $\langle cb^ma, a \rangle$ and next $\langle cb^m, b \rangle$. The illegal turn
is never a yo-yo since there is no cut edge in the graphs along the paths.
The loop labeled $b$ in $y$ is legal throughout and hence is maximally stretched from 
$y$ to $w$ and from $w$ to $z$. Therefore $w$ lies on a rigid folding path from 
$y$ to $z$. By \thmref{Thm:Unique} the folding path is the unique (up to re-parametrization) 
geodesic from $y$ to $z$.

We now compute distance to the center of the ball. For large enough $m$, 
we have 

 \begin{align*}
 d(y, x) &= \log \frac{|cb^m a|_x}{|cb^m a|_y} 
             = \log \frac{\frac{1}{2}+ 1 + \frac{1}{2} - \frac 1 m }{\frac{m+2}{2m+4}}
             = \log \frac{4m^2+ 6m -4}{m^2+2} 
              < \log 5 < 2 \\
  d(w, x) &= \log \frac{|a|_x}{|a|_w} 
              = \log \frac{\frac1 2 - \frac 1 m }{\frac 1 {m+4}}
              = \log \frac{m^2 + 2m -8}{2m} \geq \log \frac m2\\
  d(z, x) &= \log \frac{|c|_x}{|c|_z} = \log \frac{\frac 12}{\frac 13} = \log \frac 32 < 2.  
 \end{align*}
That is, $y, z \in \Bi(x, 2)$ and the distance $d(w, x)$ can be made to be arbitrarily 
large. 
\end{proof}

\bibliographystyle{alpha}

\end{document}